\documentclass[12pt]{article}

\usepackage{amssymb,amsmath,amsfonts,amssymb}
\usepackage{graphics,graphicx,color,hhline}
\usepackage{bbm}
\usepackage{stmaryrd}
 \usepackage{mathtools}

\usepackage{authblk}
\usepackage{euscript}

\def\@abssec#1{\vspace{.05in}\footnotesize \parindent .2in 
{\bf #1. }\ignorespaces} 

\graphicspath{{/EPSF/}{Figures/}}   

\newtheorem{theorem}{Theorem}[section]
\newtheorem{lemma}[theorem]{Lemma}

\def \Rm {\mathbb R}

\def \Cm {\mathbb C}

\def\un{{\mathbbmss{1}}} 

\newcommand{\eps}{\varepsilon}
\newcommand{\E}{\mathbb E}

\newcommand{\be}{\begin{equation}}
\newcommand{\ee}{\end{equation}}
\newcommand{\bea}{\begin{eqnarray}}
\newcommand{\eea}{\end{eqnarray}}
\newcommand{\bee}{\begin{eqnarray*}}
\newcommand{\eee}{\end{eqnarray*}}

\newcommand{\bu}{\mathbf u} \newcommand{\bv}{\mathbf v}

\newcommand{\bx}{\mathbf x} \newcommand{\by}{\mathbf y}

\def\fref#1{{\rm (\ref{#1})}}

\newcommand{\calH}{\mathcal H}

\newcommand{\calO}{\mathcal O}
\newcommand{\calK}{\mathcal K}

\newcommand{\calR}{\mathcal R}

\newcommand{\calG}{\mathcal G}
\newcommand{\calJ}{\mathcal J}

\newcommand{\cout}[1]{}

\usepackage[normalem]{ulem}
\usepackage{soul}
\usepackage{xcolor}

\begin{document}
\title{Blind source separation for imaging}
\author[1,2]{Randy Bartels\footnote{rbartels@morgridge.org}}
 \author[3]{Olivier Pinaud\footnote{olivier.pinaud@.colostate.edu} }
  \affil[1]{Morgridge Institute for Research, Madison, WI 53715 USA}
 \affil[2]{Biomedical Engineering Department, University of Wisconsin--Madison, Madison, WI 53715 USA}
 \affil[3]{Department of Mathematics, Colorado State University\\ Fort Collins CO, 80523}
\maketitle
\begin{abstract}
  This work is concerned with the problem of blind source separation and its applications to imaging. We first establish a theoretical result that we stated in our previous article \cite{BPV25} on imaging in diffusive environments. This result is a generalization of separability criteria found in the literature to arbitrary correlated complex-valued sources with additive noise. In a second step, we verify these separability conditions in two propagation regimes frequently encountered in imaging: the speckle regime and the random geometrical optics regime. Finally, we propose a new imaging method based on the blind source separation problem that improves on images obtained with the classical decomposition of the time reversal operator method.  
   \end{abstract}

   \section{Introduction}
   This article is the companion paper of our work \cite{BPV25} where we proposed a new method for imaging biological samples in a highly diffusive regime. A key ingredient in the technique is the resolution of a Blind Source Separation problem (BSS) allowing us to extract, from measurements of the total field, individual Green's functions associated with point scatterers. Once estimated, these Green's functions are used to create an image by solving a deconvolution problem. Specifically, in \cite{BPV25} we solved the BSS problem using classical methods of Independent Component Analysis (ICA), in particular the RobustICA algorithm and associated code \cite{zarzoso2009robust}. The standard BSS setting considers a set of \textit{independent} sources, and the theoretical analyses of ICA methods found in the literature strongly rely on the assumption of independence. The latter is not verified in our imaging problem because the sources correspond to scatterers that yield correlated fields when they are in close proximity of each other. Moreover, when these sources are complex-valued, the classical theoretical setting assumes that the sources are \textit{circular}, namely random variables of the form $X+iY$ where $\E\{X\}=\E\{Y\}=\E\{XY\}=0$, $\E\{X^2\}=\E\{Y^2\}$, which may or may not be satisfied in practice. In particular, the circularity condition holds only approximately in our imaging problem in diffusive environments.


   Our main objectives in this work are then as follows. 
   
   First, we extend the theoretical analysis of the classical ICA functional (here the kurtosis) to \textit{correlated non-circular sources}. To this end we derive explicit separability conditions on the sources under which, together with smallness assumptions on a possible additive noise, the solutions to the BSS problem are approximately among the local extrema of the kurtosis under appropriate constraints. We also obtain an error estimate. This is a theoretical \textit{local} result that generalizes that of \cite{fastICA}, and since the optimization problem is not convex, there is in principle no guarantee that the numerical algorithms find our solution of interest without a good initial guess. Although when the sources are independent and the data are real-valued, \cite{douglasICA} shows that for the FastICA method the algorithm actually converges to solutions corresponding to separated sources, in the limit of infinite sampling (namely, the expectations can be computed exactly and are not approximated by empirical averages). (FastICA is similar to RobustICA, but with a different line search.) These conditions are not verified in our complex-valued, correlated, non-circular setting with finite-size sampling. Further, it is shown in \cite{zarzoso2009robust} that when the sample size is too small, solutions may be trapped around extrema not associated with separated solutions. Nevertheless, the results of \cite{zarzoso2009robust} hint at the fact that the RobustICA algorithm may behave well in our imaging setting. Indeed, this is what we have observed in \cite{BPV25}, where the algorithm always converged to our solution of interest. However, we observe in the present work, in a different imaging regime, situations where the algorithm does not converge to the desired solutions.


   A second goal is to verify these separability conditions for two different imaging regimes: the speckle regime of \cite{BPV25} where the sources are fully random fields, and the Random Geometrical Optics (RGO) regime where the sources are plane waves perturbed by random phases. The nature of the BSS problem is very different in these two regimes: the speckle regime is similar to the standard BSS formulation where the kurtosis is estimated with empirical averages over independent realizations of the illuminating fields, while in the RGO regime the kurtosis is computed by integration of the fields over the domain of measurement.



   Our third goal is to show an example of application of BSS to imaging. We considered in \cite{BPV25} a speckle Second Harmonic Generation (SHG) regime, where both the illuminating and the measured fields are fully random. We investigate in the present work a scenario where the illuminating fields are fully random while the measured fields are in the RGO regime. 

   
We propose a two-step method to construct images. As in \cite{BPV25}, the first step is to separate the fields emanating from the different scatterers. The second is to exploit these individual fields to remove the local random phases to obtain diffraction-limited images. Our method is an improved version of the classical decomposition of the time reversal operator (DORT) method \cite{DORT2}, which is generally only effective in weakly heterogeneous media. The source separation is the key ingredient of the method.


   We consider both SHG and linear scattering imaging scenarios. It was not possible to treat linear scattering in \cite{BPV25} since standard ICA methods exclude multiple Gaussian sources, and the incoming and measured fields are both Gaussian in a speckle regime. The situation is different in the RGO regime. Under appropriate conditions, we show that it is possible to separate the Green's functions, allowing us (i) to decrease the number of illuminations needed in the SHG case, and (ii) to treat linear scattering.


   The article is structured as follows. We introduce in Section \ref{BSS} generalities on blind source separation and present our optimality result. Section \ref{sec:appli} is devoted to applications to imaging. We first define the measurement setting (a typical microscopy experiment) in Section \ref{ssec:mes}, then particularize the BSS problem to imaging in Section \ref{ssec:bss}. We verify the separability criterion obtained in our optimality result in both the speckle and RGO regimes. 
   We introduce our imaging procedure in Section \ref{secIm}, and present numerical simulations in Section \ref{sec:simu}, both for SHG and linear scattering. Finally, a proof of our main theoretical result is given in Appendix \ref{app:proof}.

   
   
 \paragraph{Acknowledgment.} This work was supported by NSF grant DMS-2404785 and grants 2023-336437 and 2024-337798 from the Chan Zuckerberg Initiative DAF, an advised fund of Silicon Valley Community Foundation.

   \section{Blind source separation} \label{BSS}

   In this section, we introduce the basic problem of BSS, and define first some notations that will be used throughout the article.
 
    \paragraph{Notation.} The canonical basis of $\Rm^N$ is denoted by $\{e_p\}_{p=1,\cdots,N}$. $\Re z$ and $\Im z$ are the real and imaginary parts of $z \in \Cm$, and both $\bar z$ and $z^*$ will denote the complex conjugate of $z$. We write $x^*$ for the conjugate transpose of a vector $x \in \Cm^N$, and depending on which one is more convenient, we will use two notations for the $\Cm^N$ inner product, $(x,y)=x^*y$ for $x,y \in \Cm^N$. Since there will be no possible confusion, we will use the notation $\| \cdot \|$ for both vector and matrix norms, and we set $ \|x\|=(x,x)$ when $x \in \Cm^N$. In the error estimates, we will not keep track of the dependency of some constants on the matrices dimensions, and as a consequence any norm can be used in the proof of  Theorem \ref{th:err}. For $A \in \Cm^{N \times N}$, we set for instance $\| A \|=\max_{i,j=1,\cdots, N} |a_{ij}|$. For two matrices $A_0$ and $A_1$, we will write $A_0=\calO(A_1)$ whenever there is a constant $C>0$, independent of $A_0$ and $A_1$, such that $\|A_0\| \le C\|A_1\|$. We will adopt the same notation when $A_0$ is a scalar or a vector.
   
          \subsection{Generalities} \label{secgene}
          We describe here the general BSS problem and particularize it to our imaging setting in Section \ref{ssec:bss}. The starting point is the linear system
          $$x=As+n,$$
          where $A$ is a (deterministic) complex-valued $N \times N $ invertible matrix, $s \equiv (s_i(\omega))_{i=1,\cdots,N}$ is a complex-valued random (column) vector of length $N$ that contains the sources, and $n$ is a vector mean-zero random noise, independent of $s$ for simplicity. The variable $\omega$ models the sampling process, and we have in practice $\omega=(\omega_i)_{i=1,\cdots,N_{sa}}$, that is we have access to $N_{sa}$ samples of the vector $x$, independent or not. For computations, the data $x$, source $s$ and noise $n$ are in general identified with $N \times N_{sa}$ complex-valued matrices.  The matrix $A$, as well as $s$ and $n$, are all unknown, and the goal is to extract $s$ from knowledge of $x$ only. This is done using ICA.

          There are natural degeneracies in the problem: the data $x$ is invariant after arbitrary permutations of the columns of $A$  and the entries of $s$, as well as after rescaling these columns and the components of $s$ appropriately. Writing
          $$
          \E \{x\}:=\frac{1}{N_{sa}}\sum_{j=1}^{N_{sa}} x(\omega_j)
          $$
for the empirical average of $x$, we suppose without lack of generality that $\E\{s\}=\E\{n\}=0$ ($0$ is meant here as a $N$ column vector) and that $\E\{|s_i|^2\}=1$ for $i=1,\cdots,N$. The best that can be achieved with such hypotheses is hence an estimate of $s$ up to permutations and rescaling of each entry by a complex number with unit modulus. 
          
          The main idea behind standard ICA is to maximize non-Gaussianity: suppose for the moment that the components $s_i$ are independent with same distributions; according to the central limit theorem, a linear combination of the $s_i$ will be closer to a Gaussian distribution than the $s_i$ individually. To separate the $s_i$, one then optimizes some criterion that characterizes Gaussianity. One such criterion used in the FastICA \cite{fastICA} and RobustICA \cite{zarzoso2009robust} algorithms is based on the kurtosis, and here we use the definition given in RobustICA: let $w \in \Cm^N$ be a column vector with unit norm. Given $x=As+n$ and writing $y=w^* x \in  \Cm$, the kurtosis considered in RobustICA has the form
$$
K(w)=\frac{\E \{|y|^4\}- 2 (\E\{|y|^2\})^2-|\E\{y^2\}|^2}{(\E\{|y|^2\})^2}.
$$
When $y$ is a complex Gaussian random variable, its kurtosis is zero in the case where $\E$ denotes a true expectation and not an empirical average.

The expression of the kurtosis is simplified when $x$ is pre-whitened, that is when $\E \{xx^*\}=I$, for $I$ the $N \times N$ identity matrix. The pre-whitening is easily accomplished with a Singular Value Decomposition (SVD), and we will assume from now on that the data is pre-whitened. With the normalization condition $w^* w=1$, we find
$$\E\{ yy^*\}=\E\{ w^* x x^* w\}=w^* w=1,$$
and $K$ becomes
$$
K(w)=\E \{|y|^4\}-2-|\E\{y^2\}|^2.
$$

ICA methods consist in finding the critical points of $K$ under appropriate constraints. By critical points, we mean minimizers and maximizers of $K$ under constraints. The sign of the kurtosis of a given source determines whether it is a minimizer (when $K<0$) or a maximizer (when $K>0$). These critical points give in turn estimates of the columns of $(A^{-1})^*$, and after finding $N$ critical points we expect to obtain an estimate of $(A^{-1})^*$, up to the degeneracies. From a theoretical standpoint, it is shown in \cite{fastICA} that when the $s_i$ are mutually independent, when at most one of the $s_i$ is Gaussian, when $s_i$ is circular, and finally when $\E\{\cdot \}$ is the true expectation and not the empirical average as we have here, that the columns of $(A^{-1})^*$ are among the critical points of $K(w)$ under appropriate constraints (that are discussed further). More can be said when the problem is real-valued, in which case it is shown in \cite{douglasICA} that the FastICA algorithm \cite{fastICA} yields $(A^{-1})^*$ (as always, up to permutations and rescaling of the columns).

In our case, the $s_i$ for different $i$ are in general not independent since the fields associated with nearby scatterers are correlated. As stated in the Introduction, one of our goals in this work is to extend the above optimality result to the case of dependent non-circular sources. Before stating our main theoretical result, we need to be more precise about the optimization problem. 

ICA algorithms proceed by iterations: at the first step, $K(w)$ is optimized under the constraint $\|w\|=1$, which is expected to yield an estimate of one column of $(A^{-1})^*$. At the following steps, the kurtosis is optimized under conditions that prevent the algorithm from converging to a previously obtained critical point. Natural constraints would be to impose orthogonality between the critical points, that is that if $w_1 \in \Cm^N$ is the critical point obtained at the first step, the constraints
$$
\|w\|=1, \qquad w^*w_1=0,
$$
are imposed at the second step. This is the standard approach of \cite{fastICA}, but is problematic when the sources are correlated since the estimated matrix $A$ will be unitary by construction while the true $A$ might not be. Consider indeed the noise free case where $n=0$.  The whitening condition yields
$$
I=A C^{(s)} A^*, \qquad \textrm{where} \quad C^{(s)}=\E \{s s^*\},
$$
and as a consequence $A$ is unitary only when the correlation matrix $C^{(s)}$ is the identity matrix $I$, that is when the sources are decorrelated. 

As mentioned in \cite{zarzoso2009robust}, deflation techniques provide an alternative and proceed as follows:  denote by $(a_i)_{i=1\cdots,N}$ the columns of $A$ and assume without lack of generality that the first column of $A^{-1}$ was estimated at the first step. This provides us in turn with an estimate $\tilde{s}_1$ of the source $s_1$ up to a multiplicative factor. Deflation is done by considering
$$
h_1=\underset{h \in \Rm^N}{\textrm{argmin}}\; \E \{ \| x-h \tilde{s}_1 \|^2\}=\frac{\E\{x (\tilde{s}_1)^*\}}{\E\{|\tilde s_1|^2\}}.
$$
The data $x$ is subsequently modified as
\bee
x \to x^{(2)}&=&x-h_1 \tilde s_1\\
&=&\sum_{j=2}^N a_js_j+n-\frac{\tilde s_1}{\E\{|\tilde s_1|^2\}} \left(\sum_{j=1}^N a_j \E\{s_j (\tilde s_1)^*\}+\E\{n (\tilde s_1)^*\}\right)+a_1 s_1.
\eee
The kurtosis $K(w)$ is then optimized under the constraint $w^*w=1$ with the new data $y=w^* x^{(2)}$. At step $\ell \geq 2$, the data has the form
$$
x^{(\ell)}=\sum_{j=\ell}^N a_js_j+n+n^{(\ell)},
$$
where
$$
n^{(\ell)}=\sum_{i=1}^{\ell-1} a_i s_i-\sum_{i=1}^{\ell-1} \frac{\tilde s_i}{\E\{|\tilde s_i|^2\}} \left(\sum_{j=1}^N a_j \E\{s_j (\tilde s_{i})^*\}+\E\{n(\tilde s_i)^*\}\right).
$$
Hence the optimization problem at each step has the same form, what changes are the expressions of the linear systems and of the noise. We will then only pursue the analysis for the first step since it carries over immediately to an arbitrary iteration $\ell$ provided the data $x$ is replaced by $x^{(\ell)}$ and the noise by  $n+n^{(\ell)}$.

\subsection{Optimality result} \label{opti} Consider without lack of generality the first column of $(A^{-1})^*$, that we denote by $c_1 \in \Cm^N$. Let ${w}_0=c_1/\|c_1\|$. With such a defined $w_0$, we have $w_0^* A= \gamma e_1^*$, where $\gamma=\|c_1\|^{-1}$ ($e_1$ is the first vector of the canonical basis of $\Cm^N$). This can be recast as $A^*w_0=\gamma e_1$. Our goal is to show that $w_0$ is approximately a critical point of $K$ with unit norm constraint. We will use the angular distance on the complex unit sphere defined by
$$
\textrm{dist}(x,y)=\arccos (\Re (x,y)), \qquad x,y \in \Cm^N.
$$
The optimization problem is invariant upon multiplication by a global phase factor since $K(w)=K(w e^{i \theta})$ and the unit sphere remains unchanged. The critical points of $K$ will then be at best estimates of the columns of $(A^{-1})^*$ up to global phase factors. This is not an issue for the imaging problem, but is too vague for deriving an error estimate in terms of the distance on the unit sphere of $\Cm^N$.

Since obviously $(w_0,w_0)=1$, it is natural to consider optimizers satisfying the constraint $\Re (w,w_0) \geq 0$ if our goal is to obtain an approximation of $w_0$. Let then $\widetilde{w}_e$ be a global optimizer of $K(w)$ (a minimizer if the kurtosis of $s_1$ is negative, a maximizer if it is positive) on the unit sphere with constraint $\Re (w,w_0) \geq 0$. This latter condition simply means that we are looking for critical points of $K$ on the ``northern hemisphere'' of the unit sphere with north pole $w_0$. In order to compare such an optimizer with $w_0$, we need to choose the optimizer $w_e$ that yields the equality $(w_e,w_0)=1$ when there is no noise and the sources are independent. This is achieved by defining $w_e=\widetilde{w}_e \exp(i \arg(\widetilde{w}_e,w_0))$, so that
$$
(w_e,w_0)= |(\widetilde{w}_e,w_0)|.
$$

We will then derive an error estimate between $w_e$ and $w_0$, showing that $w_e$ is approximately a local minimizer of $K$ on the unit sphere. For this, the correlation matrix of the noise $n$ is denoted by
$$
C^{(n)}=\E \{n n^*\}.
$$
We next introduce the following quantities:
\begin{align} 
  &M^{(s)}=\max_{j\neq 1} \Big|\E\{|s_j|^2 |s_1|^2\}-\E\{|s_j|^2\} \E\{|s_1|^2\}\Big|+\Big|\E\{s_j^2 (s_1^*)^2\}-\E\{s_j^2\} \E\{(s_1^*)^2\}\Big|\label{not1} \\ \nonumber
  &\hspace{1.3cm} +\max_{\tiny{\begin{array}{l}i,j \neq 1\\ i \neq j \end{array}}} \Big|\E\{s_js_i (s_1^*)^2\}-\E\{s_js_i\} \E\{(s_1^*)^2\}\Big|+\Big|\E\{s_js_i^* |s_1|^2\}-\E\{s_js_1\} \E\{s_i^* s_1^*\}\Big| \\ \nonumber
  &\hspace{1.3cm}
  +\|C^{(s)}-I\|\\[3mm]
  &M^{(n)}=\max_{j} \; \E \{ n_j^2\}+\E \{ n_j^4\}.\label{not2}
\end{align}
Above, we recall that $C^{(s)}$ is the correlation matrix of $s$. The constant $M^{(s)}$ is a measure of independence of the sources, and $M^{(n)}$ of how large the noise is. 

In the next theorem, we estimate the distance between $w_e$ and $w_0$ on the complex unit sphere. 

\begin{theorem} \label{th:err}  Assume that $\E \{|s_1|^4\}-2-|\E\{s_1^2\}|^2\neq 0$. Then, there exists $\eps>0$ such that the estimate
  $$
  \textrm{dist}(w_e,w_0)=\calO\big([M^{(s)}+M^{(n)}]^{1/4}\big)
  $$
  holds under the condition $M^{(s)}+M^{(n)} \leq \eps$.
\end{theorem}

The condition $\E \{|s_1|^4\}-2-|\E\{s_1^2\}|^2\neq 0$ is equivalent to the empirical kurtosis of the source $s_1$ not being zero, which holds  when $s_1$ is not Gaussian. Under such a condition, the theorem shows, provided that the sources are sufficiently independent and the noise is sufficiently small, that the critical point $w_e$ is an approximation of $w_0$. The theorem provides us with an explicit criterion that can be tested in various configurations. In Sections \ref{ssec:speckle} and \ref{ssec:sepRGO}, we estimate the constant $M^{(s)}$ in the speckle and RGO regimes, respectively. Note, as expected, that $M^{(s)}=0$ in the infinite sampling limit when the sources are independent. 

The theorem gives us theoretical separability conditions but does not say anything about the convergence of a numerical algorithm towards $w_e$. As already mentioned, it is though expected, provided $M^{(s)}+M^{(n)}$ is sufficiently small, that gradient-type algorithms such as FastICA or RobustICA may be able to find $w_e$ without a good initial guess. 


\section{Applications to imaging} \label{sec:appli}
This section is devoted to the applications of BSS to imaging. We first introduce how measurements are performed.
  \subsection{Measurement setting} \label{ssec:mes}
  We suppose we are in the same standard holography setting described in detail in \cite{BPV25}: a sample is illuminated by a laser going through a microscope with focal length $z_s$, and the scattered field is measured at a camera following an interferometric procedure. The sample is placed at one conjugate plane of the microscope while at the other rests a Spatial Light Modulator (SLM) allowing one to control the phase of the field on this plane; see Figure \ref{fig1}. Without scattering media between the microscope and the sample, the field at one conjugate plane is, according to standard Fourier optics, the 2D Fourier transform of the field at the other conjugate plane; see e.g. \cite{mertz2019introduction}, Chapter 3. 
\begin{figure}[h!]
\begin{center} 
  \includegraphics[scale=1.2]{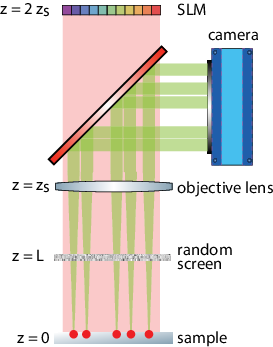}
  \end{center}
  \caption{Measurement setting} 
  \label{fig1}
\end{figure}



Scattering is often modeled in experiments by random phase screens consisting of ground glass diffusers. We place one random screen at $z=L$, and various regimes can be obtained depending on the fluctuations of the  phase. Here, the fluctuations are set such that the outgoing field is in the RGO regime (details below). A second screen is added in \cite{BPV25} in the speckle regime. 

  The measured data are stored into the so-called reflection matrix $\calR$ (see e.g. \cite{GiganTR}), which has the following form in the present measurement setting
    $$
     \calR_{ij}=\int_{\Rm^2} \calG(\bu_i;\bx)    V_s(\bx)  \calH^j(\bx) d\bx+n_{ij}.
     $$
     See \cite{BPV25} for more details. Briefly, above $n_{ij}$ is some additive noise, $(\bu_i)_{i=1,\cdots,N_p}$ denotes the collection of pixels on the camera where the field is measured, and the index $j=1,\cdots,N_r$ refers to the label of the illumination; $\calH^j(\bx)$ denotes the $j$ incoming field illuminating the sample placed at $z=0$ along the optical axis, while $\calG(\bu_i;\bx)$ is the outgoing field emitted by a point scatterer at transverse position $\bx$ on the sample plane at $z=0$ and measured at pixel $\bu_i$ in a objective lens pupil plane located $z=2z_s$ (the microscope lens is at $z=z_s$ and $\bu$ could be reimaged to the camera). The quantity $V_s(\bx)$ (the scattering potential) is what we aim to estimate.

     In the SHG case, we have $\calH^j(\bx)=(W^j(\bx))^2$ for an input field $W^{j}$, while $\calH^j(\bx)=W^j(\bx)$ in the linear scattering case. The $N_r$ illuminations are obtained by choosing $N_r$ independent random phases on the SLM, and these phases have statistical parameters such that the $W^j$ are in the \textit{Gaussian speckle regime}. More details are given in Section \ref{sec:in}.

     
      We discretize the integral in $\calR$ using a grid $(\bx_j)_{j=1,\cdots,N_d}$ with $N_d$ points, and the reflection matrix then reads ($^T$ denotes transposition)
   
     $$
     \calR_{ij}=\sum_{\ell=1}^{N_d}\calG(\bu_i;\bx_\ell)    V_s(\bx_\ell)  \calH^j(\bx_\ell) = (\textsf{G} \rho^{(0)} \textsf{H}^T)_{ij}+n_{ij},
     $$
     where $\textsf{G}$ is a $N_p \times N_d$ matrix with components $\textsf{G}_{ij}=\calG(\bu_i;\bx_j) $, $\rho^{(0)}$ is a diagonal $N_d \times N_d$ matrix with diagonal entries $\rho^{(0)}_{jj}= V_s(\bx_j) $, and $\textsf{H}$ is a $N_r \times N_d$ matrix with components $\textsf{H}_{j \ell}=\calH^j(\bx_\ell)$. Equivalently, we can say that $V_s$ consists of $N_d$ points scatterers located at $(\bx_j)_{j=1,\cdots,N_d}$, and obtain a similar expression for $\calR$.

     In the next two sections, we model the incoming and outgoing fields using standard Fourier optics.

     \subsubsection{The incoming field} \label{sec:in}



     We follow \cite{mertz2019introduction}, Chapter 3: neglecting the finite size of the microscope lens and knowing the field $W_{\rm{SLM}}$ on the SLM at $z=2z_s$, the field at $z=L$, right before the random screen, is 
  \bea
W(\bx,L)&=&  e^{ik_0 (2z_s -L)} \int_{\Rm^2}e^{i k_0  L |\bv|^2/2z_s^2} e^{-i k_0  \bx \cdot \bv/z_s} W_{\rm{SLM}}(\bv) P_{\rm{SLM}}(\bv) d\bv, \label{WSLM}
\eea
where $k_0=2 \pi/\lambda$, for $\lambda$ the central wavelength of the laser, and $P_{\rm{SLM}}$ the pupil function, here a circular aperture with radius $R_{\rm{SLM}}$. The above integral will be evaluated in the simulations with a Fast Fourier Transform. The field $W(\bx,L)$ is then multiplied by the random phase $e^{i k_0 S(\bx)} P$, for $P$ the support of the random plate and $S$ some random field. 
The field finally propagates to $z=0$, and we use a simple paraxial (Fresnel) propagator for this, resulting in the expression 
\bea \label{UL}
W(\bx,0) &=& e^{ik_0L} \int_{\Rm^2} e^{i k_0 |\bx-\bv|^2/2L} e^{i k_0 S(\bv)}W(\bv,L) P(\bv)d\bv.
\eea
The field $W_{\rm{SLM}}$ has the form $W_{\rm{SLM}}=e^{i \sigma_{\rm{SLM}} V_{\rm{SLM}}}$, where $V_{\rm{SLM}}$ is a mean-zero random field with unit variance and correlation length $\ell_{\rm{SLM}}$. As in \cite{BPV25}, we choose $R_{\rm{SLM}}$, $\sigma_{\rm{SLM}}$, and $\ell_{\rm{SLM}}$ so that $W(\bx,0)$ is in the \textit{Gaussian speckle regime}, namely $W(\bx,0)$ is approximately a complex circular Gaussian field, i.e. a mean-zero field such that, for all $\bx$ on the sample plane,
$$\E\{(\Re W(\bx,0))^2\} \simeq \E\{(\Im W(\bx,0))^2\}, \qquad \E\{(\Re W(\bx,0))(\Im W(\bx,0))\} \simeq 0.
$$
This is a direct consequence of the central limit theorem, see \cite{BPV25}, Section 5.1, for more details. We give numerical values for $R_{\rm{SLM}}$, $\sigma_{\rm{SLM}}$, and $\ell_{\rm{SLM}}$ in Section \ref{sec:simu}.

The key quantity for imaging is the correlation length $\ell_{c,\rm{in}}$ of $W$, that is the length such that
$$
\frac{\E\{ W(\bx,0) W^*(\by,0)\}}{\E\{ |W(\bx,0)|^2\}} \ll 1 \qquad \textrm{when} \qquad |\bx-\by| \gg \ell_{c,\rm{in}}.
$$
Setting $R_{\rm{SLM}}=\textrm{NA}\; z_s$, for NA the numerical aperture of the microscope, and choosing the phase $S$ appropriately, we show in \cite{BPV25} that $\ell_{c,\rm{in}}=\lambda/2\rm{NA}$, which is the usual resolution limit in microscopy. We will see in Section \ref{ssec:speckle} that scatterers at distances greater than $\ell_{c,\rm{in}}$ can be separated in the SHG setting by BSS techniques.

We move on now to the outgoing field $\calG$.

     \subsubsection{The outgoing field} \label{sec:out}
     We model the propagation from the sample plane at $z=0$ to the plane at $z=L$ using again the Fresnel propagator. The resulting field at $z=L$ is then multiplied by the random phase $e^{i k S(\bx)} P$. In the SHG case, we have $k=2k_0$, while $k=k_0$ in the linear scattering case. The field at the camera, virtually located at $z=2z_s$, is obtained using standard holographic imaging \cite{mertz2019introduction}: neglecting again the finite size of the microscope lens, the field at $z=2z_s$ is the Fourier transform of the one at $z=L$ up to a quadratic phase factor. The mathematical formula for the Green's function at a point $\bu$ on the camera emitted by a point $\bx_j$ on the sample plane is then

        \bea \label{bG}
  \calG(\bu;\bx_j)&=&e^{ik(2z_s-L)} e^{i \frac{k L}{2 z_s^2} |\bu|^2} \int_{\Rm^2} e^{-\frac{i k }{z_s}\bu \cdot \bv} e^{ik S(\bv)}  e^{i \frac{k }{2 L} |\bv-\bx_j|^2}P(\bv) d\bv.
  \eea



Diffraction effects due to the finite size of the microscope lens will be modeled by choosing appropriately the domain on the camera used to backpropagate the data. 



  We now simplify expression \fref{bG} using a stationary phase analysis. Suppose that $P$ is the characteristic function of a disk of radius $R_s$ and that $S$ reads $S(\bx)=\sigma V(\bx/\ell_0)$, for a mean-zero stationary random field $V$ with unit variance. The parameter $\sigma$ quantifies the strength of the fluctuations and $\ell_0$ their correlation length. Under the conditions  
\be \label{contstat}
  \frac{z_s}{k R_s^2}\ll 1, \qquad \frac{\sigma L}{\ell_0 R_s} \ll 1,
  \ee
  the field $\calG$ simplifies to, up to irrelevant multiplicative factors,
   \bea \label{bG1}
  \calG(\bu;\bx_j)&=& e^{-\frac{i k }{z_s}\bu \cdot \bx_j} e^{ik S(\bx_j+\frac{L}{z_s} \bu)} P(\bx_j+L\bu/z_s).
  \eea
  This shows that, up to the localization by $P$ on the phase plate, the Green's function is approximately equal to the multiplication of free Green's function $e^{-\frac{i k }{z_s}\bu \cdot \bx_j}$, which imparts a phase tilt that maps to positions in the object $\bx_j$, and the phase factor $e^{ik S(\bx_j+L/z_s \bu)}$. This is an instance of the RGO regime particularized to our measurement setting. The first relation in \fref{contstat} is a high frequency condition, while the second one assumes that the correlation length is sufficiently large and that the fluctuation strength is sufficient small, which are  both typical of geometrical optics.



  \paragraph{Isoplanaticity.} There are two important subregimes that depend on the correlation structure of the phase screen $e^{i k S}$. First, we have when $V$ is Gaussian (see Section 5.1 in \cite{BPV25}),
 $$
\E\{e^{i k (S(\bv)-S(\bv'))}\}\simeq e^{-|\bv-\bv'|^2/l_{c}^2}
$$
where
\be \label{defLC}
l_{c}=\frac{2\sqrt{2} \ell_0}{k \sigma_0}.
\ee
The quantity $l_c$ defines the correlation length of the phase screen. We have then $S(\bx_i+\frac{L}{f} \bx) \simeq S(\bx_j+\frac{L}{f} \bx)$ when $|\bx_i-\bx_j| \ll l_{c}$. The latter condition allows us to introduce two regimes. In the \textit{isoplanatic regime}, the scatterers are such that their relative distance is less than $l_c$; in such a scenario, the Green's functions associated with different scatterers all have essentially the same random phase. The latter can sometimes be estimated, as in e.g. \cite{murray2023aberration,farah2024synthetic,CLASS}, and therefore removed to recover the free Green's function leading to improved resolution. The regime where some scatterers are separated by a distance larger than $l_c$ is called the \textit{non-isoplanatic} regime, and is much harder to deal with because multiple random phases have to be estimated to obtain good quality images. An attempt towards this goal is proposed in \cite{AubryDO} by introducing the so-called distortion operator. We propose an alternative approach in Section \ref{secIm}, namely exploiting the field separation from the BSS problem. 



     In the next section, we particularize the BSS problem to our imaging setting.

     \subsection{Blind source separation}  \label{ssec:bss}
     
     The main goal is to extract some of the columns of $\textsf{G}$ from the reflection matrix $\calR$. Similar to the DORT method \cite{DORT2}, this is accomplished by first performing an SVD of $\calR$ that yields $\calR \simeq U \Sigma V^*$, for two unitary matrices $U$ and $V$ and a diagonal matrix $\Sigma$. We keep in $\Sigma$ only $N \leq N_d$ significant singular values; 
     this allows one to filter some noise out if data are noisy, and these $N$ singular values can either correspond to $N$ well-separated scatterers or to groups of scatterers (see e.g. Section 5.5 in \cite{BPV25}). Simple calculations explicited in \cite{BPV25} show that $U$ and $V$ are formed by linear combinations of $N$ columns of $\textsf{G}$ and $\textsf{H}$, respectively, and we have
     $$
     U^*=A_0 G^*, \qquad V^*=A_1 H^T,
     $$
     for two appropriate matrices $A_0$ and $A_1$ (in the sequel we will have either $A=A_0$ or $A=A_1$). Above, both matrices $G$ and $H$ contain $N$ columns from $\textsf{G}$ and $\textsf{H}$ (in addition, the $N \times N$ diagonal matrix denoted $\rho$ further incorporates $N$ contributions from $\rho^{(0)}$), and only $U^*$ and $V^*$ are known. We would like to estimate $G$. This is a BSS problem since $A_0$ and $G$ are not known.



     A key observation made in \cite{BPV25} is the fact that even though the incoming fields $\calH^j$ may enjoy some form of orthogonality, the matrices $A_0$ and $A_1$ may not be diagonal. In different terms, the SVD alone is not sufficient to separate orthogonal fields. This is due to the well-known linear algebra fact that eigenvectors associated with degenerate eigenvalues are unstable to perturbations. And such a situation occurs in particular in our imaging setting when the scatterers have comparable intensities, i.e. $\rho_{ii} \simeq \rho_{jj}$. For this reason standard ICA techniques, which are immune to this resonance effect, are crucial.

     How to proceed depends on the structure of the data. In the speckle regime of \cite{BPV25}, the columns of $G$ are complex Gaussian random fields, which prevented us from performing BSS with data $U^*$ since standard ICA algorithms can only reconstruct at most one Gaussian source. We therefore needed to perform a separation based on the incoming fields and considered the setting of SHG holography, where the scatterers have a nonlinear second order response to the incoming wave. In such a case, we recall the field $\calH^j$ reads $\calH^j=(W^j)^2$, where $W^j$ is the $j$ illumination with central frequency $\omega_0$. The outgoing field $\calG$ has then central frequency $2 \omega_0$. With our illumination strategy based on random fields on the SLM, the ingoing fields $W^j$ are complex Gaussians as well, and $\calH^j$ is thus the square of a Gaussian field and is therefore not Gaussian. Hence, it is possible to use ICA with data $V^*$ to extract the matrix $A_1$ (up to permutations and rescaling of the columns). More precisely, we set $x=V^*$, $A=A_1$, $s=H^T$ with the notation of Section \ref{secgene}. Knowing $U$, $\Sigma$ from the SVD, and $A_1$ as a result of the BSS, then allows us to write $G \rho=U \Sigma A_1$ up to the degeneracies, and since $\rho$ is diagonal, to estimate the columns of $G$ independently. To summarize the procedure, we measure $\calR$, perform an SVD to extract the $V$ matrix, then run ICA with data $V$.



     At the theoretical level, it is possible to estimate $M^{(s)}$ in the speckle using Isserlis' theorem; details are provided in Section \ref{ssec:speckle}. 

     We also consider the RGO regime, which we have seen in Section \ref{sec:out} is very different in nature from the speckle regime. Instead of complex Gaussian random fields as in \cite{BPV25}, the outgoing Green's functions $\calG$ in the RGO regime are approximately given by free Green's functions multiplied by random phases. This regime allows us to use the matrix $U^*$ for the BSS. As claimed in the Introduction, this makes it possible in some situations to decrease the number of illuminations required for a good separation in the SHG case, and to treat linear scattering where $\calH^j=W_j$. The BSS problem is then defined with $x=U^*$, $A=A_0$, and $s=G^*$. In such a regime, we can treat pixels on the camera as sampling points, i.e. $\omega_i=\bu_i$ with the notation of Section \ref{secgene}. The matrix $x$ is of size $N \times N_p$ and the expectation $\E$ is proportional to the discretization of the integral of the measurements over the domain of the camera (denoted $D_C$ below). The expectation of the $i$ row of $s$ is hence
     \be \label{esp}
     \E\{s_i\}=\frac{1}{N_p} \sum_{j=1}^{N_p} \bar{\calG}(\bu_j,\bx_i) \sim \int_{D_C} \bar{\calG}(\bu ,\bx_i) d\bu.
     \ee
     
To summarize, we measure $\calR$, perform an SVD to extract the $U$ matrix, then run ICA with data $U$.

 
     We investigate in the next two sections the theoretical separability in the speckle and RGO regimes.
     

     \subsubsection{Separability in the speckle regime}  \label{ssec:speckle}

      We recall that  $\calH^j=(W^j)^2$, so that $s_i$, i.e. the $i$ row of $s$, is such that $s_i(\omega_j)=(W^j(\bx_i))^2$ where $\bx_i$ is the position of scatterer $i$ in the sample plane at $z=0$. Here, the input fields $W^j$ are all independent for different $j$ since they are generated by independent realizations of a random phase on the SLM. As a consequence, the symbol $\E$ in Section \ref{secgene} indeed denotes the empirical average over these realizations, and the $W^j$ are independent realizations of a random field $W$. Again, because the field $W$ is in a speckle regime, it is well-approximated by a complex circular Gaussian random field. 
     In particular, the circularity yields $\E\{(W(\bx))^2\}=\E\{(W(\bx))^4\} \simeq 0$ for any $\bx$ on the sample plane. We assume in the calculations that the number of illuminations is sufficiently large in order to treat the empirical average as the true average.


     The quantity $M^{(s)}$ defined in \fref{not1} involves second-order and fourth-order moments of $\calH^j$, and therefore fourth- and eighth-order moments of $W$. These can be computed by Isserlis' theorem exploiting the Gaussian property. We only treat here one eighth-order term in $M^{(s)}$ since the others can be estimated following the same method. A direct calculation using the circularity of $W$ and Isserlis' theorem shows that
\begin{align*}
\E&\{s_i s_j^* |s_1|^2\} \simeq 2 (\E\{W(\bx_i)W(\bx_j)^*\})^2 (\E\{|W(\bx_1)|^2\})^2\\
&+8\E\{W(\bx_i)W(\bx_1)^*\}\E\{W(\bx_j)^* W(\bx_1)\}\E\{W(\bx_i)W(\bx_j)^*\}\E\{|W(\bx_1)|^2\}\\
&+8\E\{W(\bx_i)W(\bx_1)\}\E\{W(\bx_j)^*W(\bx_1)^*\}\E\{W(\bx_i)W(\bx_j)^*\}\E\{|W(\bx_1)|^2\}.
\end{align*}
The first two terms on the right above are the leading ones since it can be shown that
$$
|\E\{W(\bx)W(\by)\}| \ll |\E\{W(\bx)W(\by)^*\}|.
$$
The dominating terms in $\E\{s_i s_j^* |s_1|^2\}$ are hence controlled by the correlation $\E\{W(\bx)W(\by)^*\}$. The correlation length $\ell_{c,\rm{in}}$ of $W$ was introduced in Section \ref{sec:in}, and we have $\E\{W(\bx)W(\by)^*\} \simeq 0$ when $|\bx-\by| \gg \ell_{c,\rm{in}}$. As a consequence,
$$
\E\{s_i s_j^* |s_1|^2\} \simeq 0 \qquad \textrm{when} \qquad |\bx_i-\bx_j|\gg \ell_{c,\rm{in}}.
$$
The correlation $\ell_{c,\rm{in}}$ depends on the scattering model from the SLM to the sample plane, and we have seen in the scenario that we consider here that $\ell_{c,\rm{in}}=\lambda/2 \rm{NA}$. The other terms in $M^{(s)}$ are treated in the same manner, and the conclusion is that
$$
M^{(s)} \simeq 0 \qquad \textrm{when} \qquad \min_{i,j=1,\cdots,N\; i\neq j}|\bx_i-\bx_j|\gg \ell_{c,\rm{in}}.
$$
If the noise is sufficiently small, Theorem \ref{th:err} implies therefore that the optimizer $w_e$ is a good approximation of $w_0$. Put differently, the fields associated with scatterers separated by at least a distance of order $\ell_{c,\rm{in}}$ can theoretically be separated when the noise is small. In practice, we have observed that the RobustICA algorithm always converges to well-separated solutions provided the number of illuminations $N_r$ is sufficiently large. 


\subsubsection{Separability in the random geometrical optics regime}  \label{ssec:sepRGO}
We have seen previously that setting $s=G^*$ yields, for the $i$ row of $s$, $s_i(\omega_j)=\calG^*(\bu_j,x_i)$, $i=1,\cdots,N$, $j=1,\cdots,N_p$. The Green's function $\calG$ admits expression \fref{bG1} in the RGO regime, and the symbol $\E$ is defined in \fref{esp}. We set $D_C$ to be the square of side $2L_C$ centered at zero. Let $\ell_{c,\rm{out}}=\pi z_c / kL_C$. With $L_C= \textrm{NA}\, z_s$, we have $\ell_{c,\rm{out}}=\lambda/2 \rm{NA}$ in the linear case and $\ell_{c,\rm{out}}=\lambda/4 \rm{NA}$ in the SHG case. The expression of $\ell_{c,\rm{out}}$ is deduced from \fref{expg} in Section \ref{secIm}.

We now plug the expression of $s$ in the definition of $M^{(s)}$ given in \fref{not1}, and exploit the orthogonality of the free Green's function. (We neglect the effects of the random phases, owing to the fact that these are perturbations and make the overall picture worse.) There are five terms to investigate.


The first one on the r.h.s. of  \fref{not1} is equal to zero since $|s_i(\omega_j)|$ is constant according to \fref{bG1}. Regarding the second term, simple calculations (see \fref{expg}) show that terms of the form $\E\{s_j^2\}$ are small provided $|\bx_j| \gg \ell_{c,\rm{out}}$. Similarly, terms of the form $\E\{s_j^2 (s_1^*)^2\}$ are small provided $|\bx_j-\bx_1| \gg \ell_{c,\rm{out}}$. As a conclusion, the second term in the r.h.s. of  \fref{not1} is small provided the scatterers are separated by a distance of at least $\ell_{c,\rm{out}}$ and that $\bx_1$ is away from the origin by at least the same length.

Consider the first contribution of the third term, that is $\E\{s_js_i (s_1^*)^2\}$. A 3-scatterer resonance phenomenon occurs when $|\bx_j+\bx_i-2\bx_1| \ll \ell_{c,\rm{out}}$, that is when $\bx_1$ is approximately $(\bx_i+\bx_j)/2$. The resonance can, in principle, prevent the algorithm from converging to a separated solution. In practice, the picture is actually more complicated: depending on the initial condition, it may happen that scatterer $\bx_i$ or $\bx_j$ is treated in the deflation before scatterer $\bx_1$, and therefore that the resonance cannot occur since one scatterer needed for the resonance is not present any longer. We have observed both scenarios in the simulations for a resonant configuration: convergence to a separated solution for some initial conditions, and convergence to non-separated solutions for other initial conditions.

The second contribution in the third term in $M^{(s)}$ is the product $\E\{s_js_i\} \E\{(s_1^*)^2\}$, $i,j \neq 1$, $i \neq j$. The latter is small when at least $|\bx_i+\bx_j| \gg \ell_{c,\rm{out}}$ or $|\bx_1| \gg \ell_{c,\rm{out}}$ are satisfied.

The fourth and fifth terms are small when
$$
\min_{i \neq j} |\bx_i-\bx_j| \gg \ell_{c,\rm{out}}, \qquad i,j=1,\cdots,N. 
$$
Indeed, it is clear that $\E\{s_js_i^*\}$ and $\|C^{(s)}-I\|$ are small when the latter condition is verified, and so is $\E\{s_js_1\} \E\{s_i^* s_1^*\}$ for $i \neq j$, $i,j\neq 1$: when $\min_{i } |\bx_i+\bx_1| \gg \ell_{c,\rm{out}}$, then $\E\{s_js_1\} \E\{s_i^* s_1^*\}$ is small, and when there is a $j_0$ such that $|\bx_{j_0}+\bx_1| \sim \ell_{c,\rm{out}}$, then $|\bx_i+\bx_1| \geq ||\bx_i-\bx_{j_0}|-|\bx_{j_0}+\bx_1|| \gg  \ell_{c,\rm{out}}$, implying that $\E\{s_i^* s_1^*\}$ is negligible when $i \neq 1$ and $i\neq j_0$.

In summary, when omitting the 3-scatterer resonance condition that we already discussed, the source $s_1$ can theoretically be extracted in a noiseless case provided (i) all scatterers are at least a distance $\ell_{c,\rm{out}}$ apart, and (ii) scatterer $\bx_1$ is at least a distance  $\ell_{c,\rm{out}}$ from the origin. These conditions have to then be adapted to the scatterers that are left at a given step in the deflation. The simulations clearly exhibited failed separation when these conditions were not met.

\subsection{Imaging procedure} \label{secIm}

We first describe the classical DORT method and then propose an improved version taking advantage of the source separation.

  \paragraph{The DORT method.} The DORT method exploits the matrix $U$ obtained from the SVD of the reflection matrix.  As claimed in Section \ref{ssec:bss}, the columns of $U$ (denoted $U_i$) are linear combinations of the columns of $G$ (denoted $G_i$), that is
  $$
U_i=\sum_{j=1}^N \alpha_{ij} G_j, \qquad i=1,\cdots,N,
$$
for some unknown coefficients $\alpha_{ij}$. There are ideal situations where the combination consists only of one Green's function, that is when the $\alpha_{ij}$ are all zero but one for a given $i$. This occurs for instance when the scatterers are sufficiently well separated and their intensities are sufficiently different. However, the resonance phenomenon described in Section \ref{ssec:bss} means the ideal situation is not a general one. While the method was originally developed to perform \textit{selective focusing}, i.e. to focus the field on particular scatterers, it is possible to follow a similar approach for imaging by taking inverse Fourier transforms of the $U_i$. Because it is not generally known how many or which fields are part of the linear combination, an image formed with one column will present one or several peaks. According to \fref{bG1}, the sharpness of these peaks depends on the measurement setting and on the strength of the fluctuations of the phase $S$. One (normalized) image is then obtained with
$$
 \widetilde J_{j}(\by)=\int_{D_C} U_j(\bu)e^{i k \by \cdot \bu/z_s}d\bu, \qquad J_j(\by)= \frac{\widetilde{J_j}(\by)}{\max_\bx |\widetilde J_j(\bx)|},
 $$
 where we recall that $D_C$ is the domain on the camera where measurements are performed. The image is noisy when the fluctuations of $S$ are too strong, and a filter $F$ can be applied to remove part of the noise. It is not our goal to optimize the choice of the filter as we expect they all perform qualitatively the same, and we then set for simplicity
      \be \label{defFIlt}
      F(g(\by))=\un\Big(\{\by \in D_C: |g(\by)| > \eta_F \max_{\bx} |g(\bx)|\}\Big),
      \ee
which filters out points with intensity less than $\eta_F$ times the maximal intensity of the image. The threshold $\eta_F$ is chosen between 0 and 1 and $\un(A)$ denotes the characteristic function of the set $A$.
In order to build an image that exhibits all scatterers, we sum the different filtered images to get
 $$
J(\by)=\left| \frac{1}{N} \sum_{j=1}^N J_j(\by) F(J_j(\by))\right|.
$$

\paragraph{The improved DORT method.} We improve the DORT method by exploiting the source separation. Assume the matrix $G$ was estimated using BSS up to the degeneracies. The $j$ column of $G$ gives access to the Green's function $\calG(\bu,\bx_j)$, where $\bu$ runs through the pixels on the camera. It is key here that the BSS yields the Green's functions individually and not as potential linear combinations, as in DORT. 




Based on expression \fref{bG1}, we define the following function
      \be \label{defg}
      g_{ij} (\by)=\int_{D_C} \calG(\bu,\bx_i)\bar{\calG}(\bu,\bx_j) e^{i k \by \cdot \bu/z_s}d\bu.
      \ee
      We recall that $D_C$ is the square of side $2L_C$ centered at zero. When $S=0$, we find
      \begin{align} \nonumber
      g_{ij} (\by)&=\int_{D_C} e^{i k(\by-(\bx_i-\bx_j)) \cdot \bu/f}d\bu\\
      &= -4 L_C^2 \; \textrm{sinc}\left(\frac{ k L_C(\by-(\bx_i-\bx_j)) \cdot \mathbf{e}_1}{z_s}\right)\textrm{sinc}\left(\frac{ kL_C(\by-(\bx_i-\bx_j)) \cdot \mathbf{e}_2}{z_s}\right) \label{expg}
      \end{align}
      where $\bu=u_1 \mathbf{e}_1+u_2 \mathbf{e}_2$ and $\textrm{sinc}(x)=\sin x/x$. In this case, the function $g_{ij}$ presents a peak at $\by=\bx_i-\bx_j$, with width $\ell_{c,\rm{out}}=\pi z_s/ kL_C$, indicating the relative position between scatterer $i$ and scatterer $j$. 

      When $S$ does not vanish, but $\bx_i$ and $\bx_j$ are sufficiently close, i.e. $|\bx_i-\bx_j| \ll l_c$ ($l_c$ previously defined in \fref{defLC}), then $S(\bx_i+\frac{L}{f} \bx) \simeq S(\bx_j+\frac{L}{f} \bx)$ and the random phases in \fref{defg} compensate each other. This is the reason for the definition of \fref{defg}. 
      When $\bx_i$ and $\bx_j$ are far apart, then $g_{ij}$ essentially consists of a noisy background without indicating the relative position $\bx_i-\bx_j$.

      The peak in $g_{ij}$ is already sharper than those in $J_i$ since part of the randomness is filtered out. Since we know there is just one peak in $g_{ij}$ because of the separation, the filtering procedure is actually very efficient to sharpen the peak even more. This is not always the case with DORT, as the number of peaks is unknown beforehand and the images are substantially noisier, resulting in the possibility that some of the scatterers' peaks are wrongly filtered out.   


      We hence apply the same filter $F$ (defined in \fref{defFIlt}) to $g_{ij}$, and denote by
      $$f_{ij}(\by)=F(g_{ij}(\by))g_{ij}(\by)$$
      the resulting image. 



      
      For $i=1,\cdots,N$, the procedure is repeated for all indices $j=1,\cdots,N$ such that $j>i$. We then build an image following the same method employed in \cite{GiganNMF}: let 
    \be \label{defsig}
\beta_{ij}=\max_{\by} |f_{ij}(\by)|, \qquad \beta_i=\max_{j=1,\cdots,N,\; j\neq i} \beta_{ij}.
\ee
Scatterers $j$ that are close to the scatterer with index $i$ have larger $\beta_{ij}$ and tend to be better estimated than those that are farther away and as a consequence have smaller $\beta_{ij}$. We then set an arbitrary threshold $\eta_I$ to remove scatterers that are too far, and a first image $I_i$ is obtained with the formula
$$
I_i=\sum_{j=1}^N |f_{ij}| h_e(\beta_{ij}- \eta_I \beta_{ii})/\beta_{ij}, \qquad i=1,\cdots,N,
$$
where $h_e$ is the heaviside function. The procedure is repeated for all $i=1,\cdots,N$, producing a set of images $(I_i)_{i=1,\cdots,N}$. These images are finally merged as in \cite{GiganNMF} by estimating the relative spatial shift between them: let $\by_{ij}^*$ the point where $\beta_{ij}$ in \fref{defsig} is achieved, and suppose our first image is $I_1$; the image that will be merged with $I_1$ is the one producing the largest $\beta_{1j}$ in \fref{defsig}. Let $j_1$ be such that $\beta_1=\beta_{1j_1}$; image $I_{j_1}$ is then shifted by $\by^*_{i j_1}$ and added to $I_1$. The procedure is repeated starting from $I_{j_1}$ until all indices are exhausted, yielding an image $I(\by)$.





The image resolution is limited by both the quality of the separation and the sharpness of the peak in $g_{ij}$. Omitting the 3-scatterer resonance, we have seen that scatterers can theoretically be separated when they are at a distance of at least $\ell_{c,\rm{in}}$ when using $V$, and at least $\ell_{c,\rm{out}}$ when using $U$. The peak in $g_{ij}$ has width $\ell_{c,\rm{out}}$. The resolution is then given by $\max(\ell_{c,\rm{in}},\ell_{c,\rm{out}})$ in the $V$ case and $\ell_{c,\rm{out}}$ in the $U$ case.

We conclude by adding that this procedure is only able to estimate the relative positions of the scatterers, not their intensity nor their absolute positions on the sample plane.


\subsection{Simulations} \label{sec:simu}

The incoming and outgoing fields of Sections \ref{sec:in} and \ref{sec:out} are computed as follows. The $\by$ integral in the SLM plane in \fref{WSLM} is evaluated with an fast Fourier transform (FFT): the domain of integration is chosen so that the variable $\bx$ (here the dual Fourier variable of $\bu=\by/\lambda z_s$) belongs to the square of side length $L_D$ centered at zero. If the latter is discretized with a stepsize $h$, then the Fourier transform in the variable $\bu$ is calculated over a square of length $1/h$ centered at zero. We set $h=L_D/2^{11}=L_D/2048$.

The convolution in \fref{UL} and the Fourier transform in \fref{bG} are computed with FFTs over the discretized square of side length $L_D$ centered at zero and stepsize $h$. The random fields $V_{\rm{SLM}}$ and $V$ are generated with Fourier series; see \cite{BPV25} for more details.

Table \ref{tab} gathers the values of some parameters used in the simulations (where $\lambda=1$).
\begin{table}[h!]
\begin{center}
\begin{tabular}{|c|c|c|c|c|c|c|c|c|c|c|c|c|c|}
    \hline
     $L_D$ & $h$ &$z_s$ & $L$& $L_C$ & $R_s$& $R_{\rm{SLM}}$ &$\ell_0$& $\ell_{\rm{SLM}}$ & $\sigma_{{\rm SLM}}$& $\eta_I$  \\
     \hline
     500$\lambda$ & $L_D/2048$ & 500$\lambda$ & 0.25$z_s$ & NA$z_s$& 0.45$L_D$&NA$z_s$ &$4\lambda$  & 2$\lambda$ & 2 &0.2\\
     \hline
\end{tabular}
\end{center}
\caption{Parameters used in the simulations.}
\label{tab}
\end{table}
\subsubsection{SHG imaging}

\paragraph{Separation with $V$.} We place $N=50$ identical scatterers at random uniformly in a square of side length $12 \lambda$ and illuminate them with $N_r=10000$ fields. The strength of the fluctuations on the random screen are set such that $k_0 \sigma=1$. In this case, the size of an isoplanatic patch is $l_c=4 \sqrt{2}\lambda\simeq 5.66 \lambda$. This is a non-isoplanatic imaging scenario since the scatterers are spread over two patches. We set as well NA=0.75. The pixel size on the image is then $\lambda/4\textrm{NA}\simeq 0.33 \lambda$.


We use the RobustICA algorithm to perform the separation and represent in Figures \ref{fig:SHGV} and \ref{fig:SHGV2} the images obtained with both the improved and regular DORT methods. We set a strong filter in Figure \ref{fig:SHGV} with $\eta_F=0.99$ and observe that the improved DORT method properly locates all scatterers while the DORT method fails to do so. Indeed, the positions shown by DORT are actually largely noise. This is clear based on Figure \ref{fig:SHGV2}, which depicts the evolution of the images as the filter strength $\eta_F $ is increased. When $\eta_F=0.5$, the scatterers do not stand out against background noise in the standard DORT image as they do in the images obtained with the improved DORT method. As $\eta_F$ increases, the filter removes the noisy background in the improved DORT images while it removes both background and scatterers in the DORT image.

This is the key mechanism at the core of the images obtained with the improved DORT method: the non-filtered images are already better since part of the noise is removed by the phase compensation in \fref{defg}. In addition, the filtering procedure is very efficient since there is just one scatterer in the images before merging and filtering, rendering the noise easier to address than with the DORT method. This is all made possible by the resolution of the source separation problem. 

An important issue with using $V$ for the separation is the large number of illuminations needed to obtain a well separated solution with RobustICA, here $N_r=10000$. It is possible to decrease $N_r$, but at the expense of separation quality. In some situations, the burden can be alleviated by using the matrix $U$ for the separation, as we pursue next.



\paragraph{Separation with $U$.}  Separation proves substantially more difficult than when using $V$: while RobustICA always does converge, it may do so to a non-separated solution. Note that it is possible to impose the sign of the kurtosis of the source (here the kurtosis is negative) in RobustICA to improve the solutions, while this was not necessary for the $V$-based separation. RobustICA is initialized with random perturbations of the identity of the form $I+M$, where $M$ is a random matrix with entries drawn uniformly between $-0.05$ and $0.05$.

We choose parameters that ensure that the algorithm always converges to separated solutions. It is possible to consider more difficult imaging scenarios, but in this case proper separation is only obtained for some initial conditions. We set $N_r=500$ and place randomly $N=10$ identical scatterers in a window of size $8 \lambda$. The fluctuations of the screen remain the same with $k_0 \sigma=1$.


As above, we see in Figure \ref{fig:SHGU} that improved DORT locates the scatterers properly while DORT does not. We do not depict the evolution as the filter strength increases as the mechanism is similar: the filter removes both scatterers and noise for DORT, while it preserves the scatterers for improved DORT. The overall cost is much lower than for the $V$-based separation since only $N_r=500$ realizations are needed.

As in \cite{BPV25}, in the full speckle regime it is possible to image more complex objects in the RGO regime. There are no particular geometrical constraints on the objects when using $V$ for the separation, the main limitation being due instead to the number of illuminations. When using $U$, objects with symmetries (such as a circle or a cross) tend to result in BSS problems that are more difficult to solve. When these symmetries are broken, by e.g. perturbing the scatterers randomly around the original shape, then the separation is in general easier. A possible explanation is that symmetries introduce non-separated extrema in the functional that the algorithm eventually converges to.  

\begin{figure}[h!]
\begin{center}
  \includegraphics[scale=0.195]{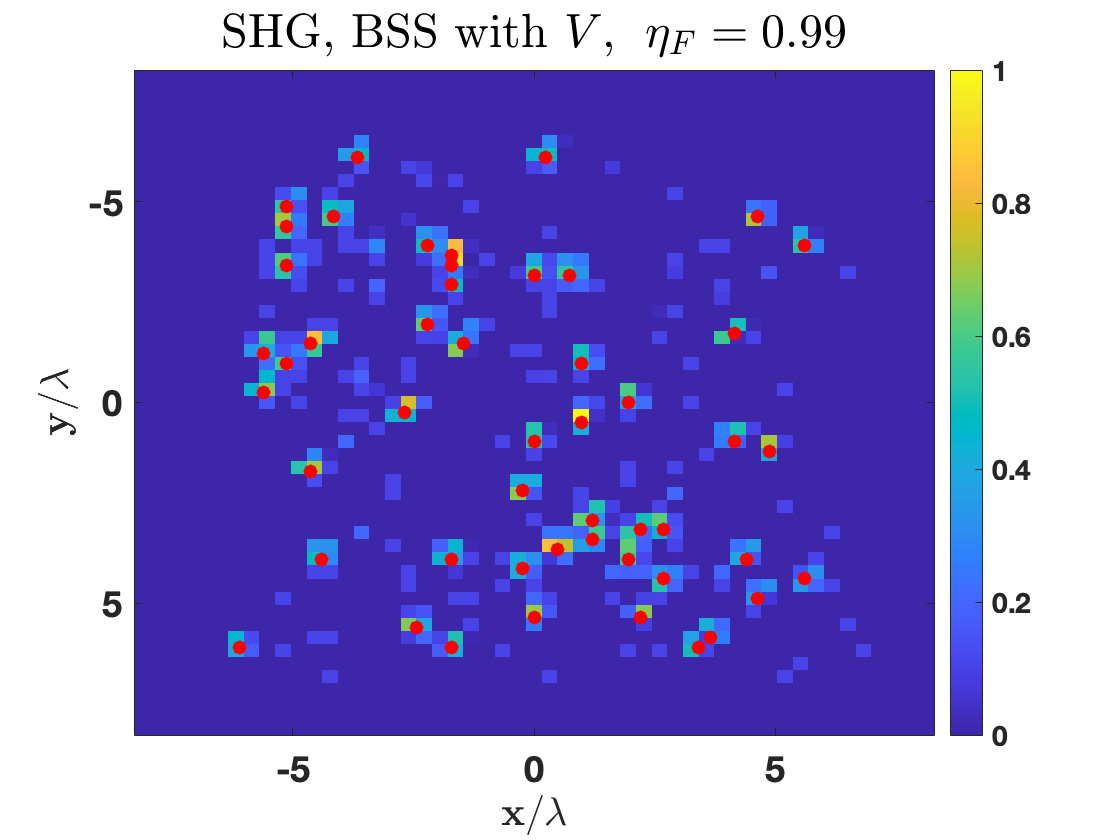}
  \includegraphics[scale=0.195]{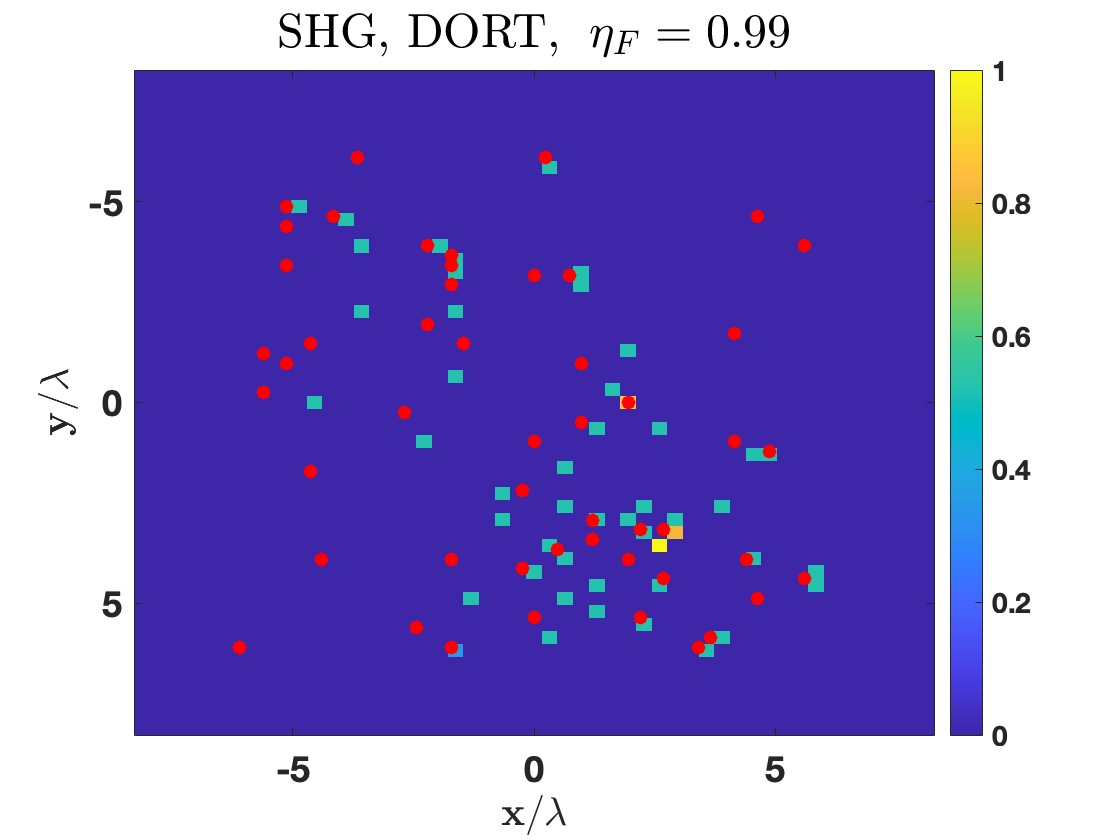} 
\end{center}
\caption{SHG imaging with BSS based on $V$. Left: Improved DORT method. Right: DORT method. The red dots represent the exact position of the scatterers.}
\label{fig:SHGV}
\end{figure}
\begin{figure}[h!]
\begin{center}
  \includegraphics[scale=0.13]{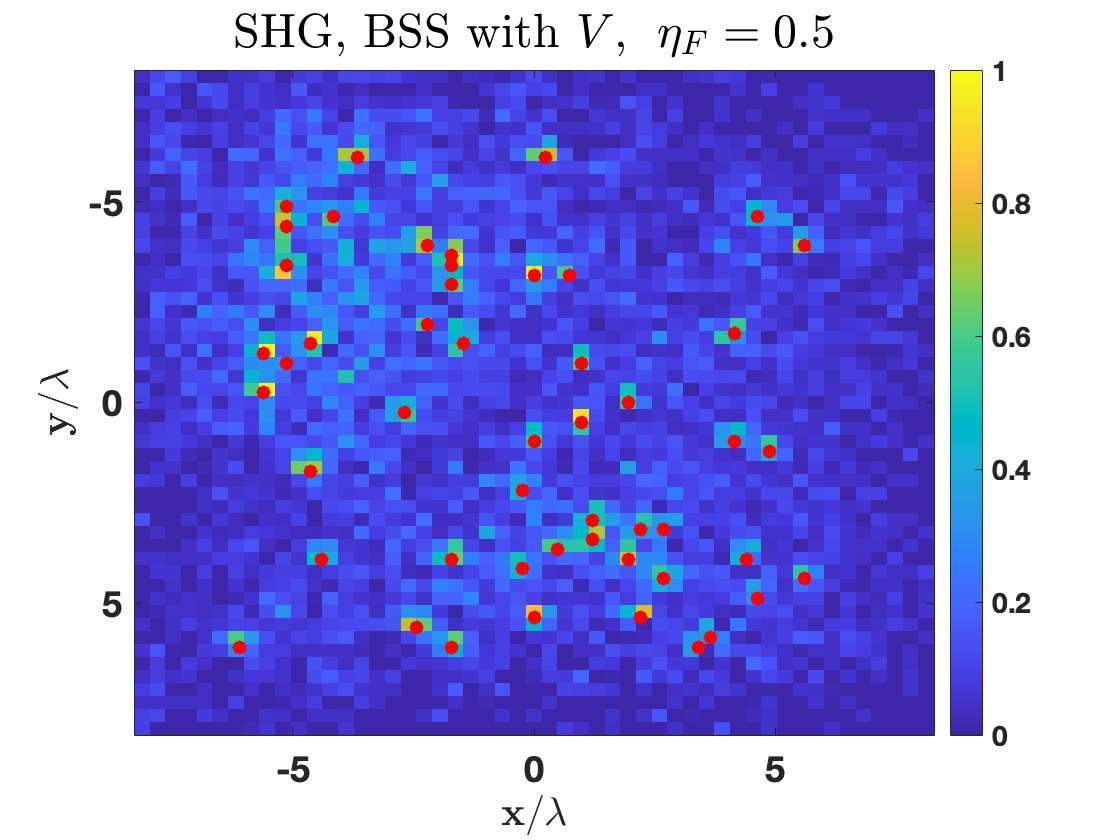}
  \includegraphics[scale=0.13]{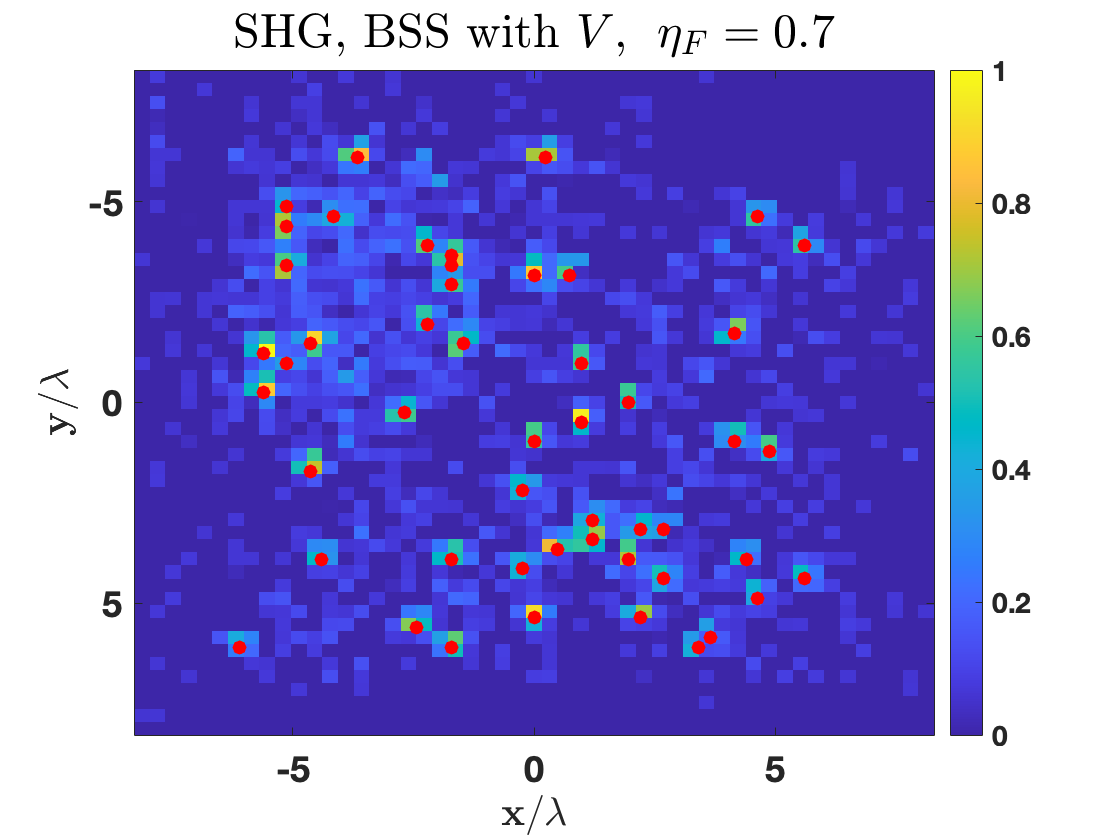}
    \includegraphics[scale=0.13]{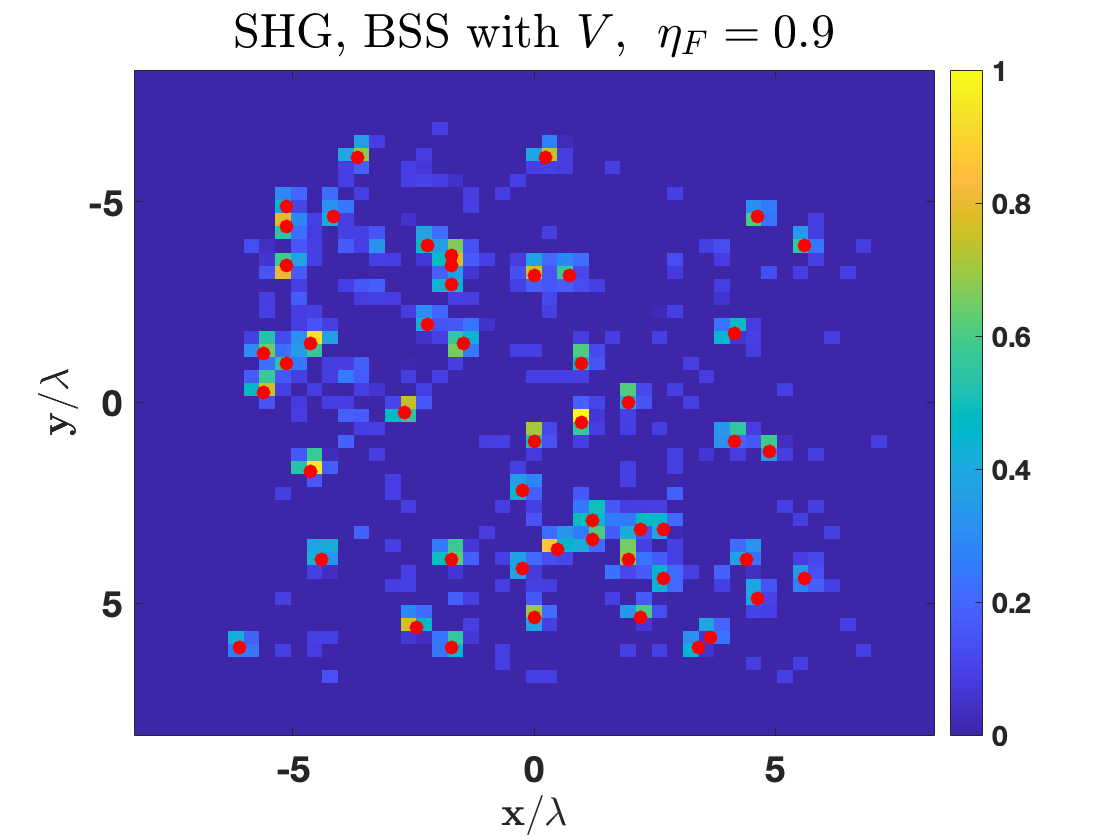} 
\includegraphics[scale=0.13]{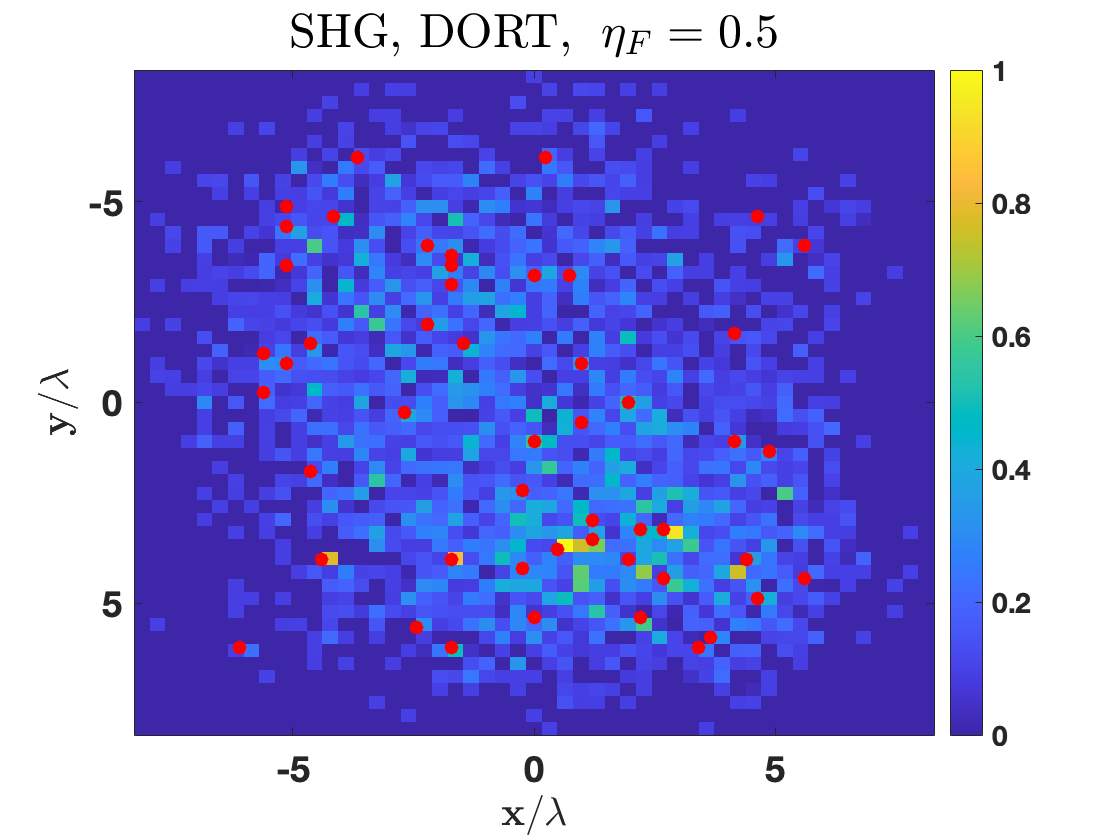}
  \includegraphics[scale=0.13]{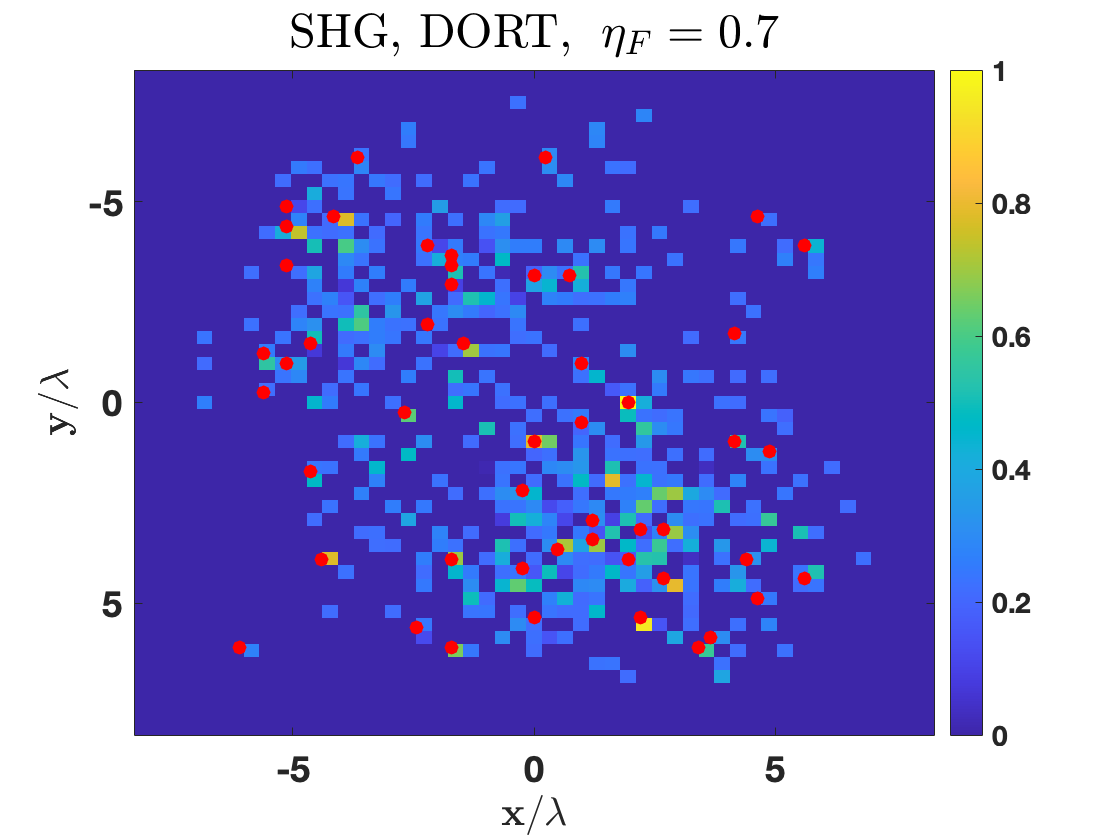}
    \includegraphics[scale=0.13]{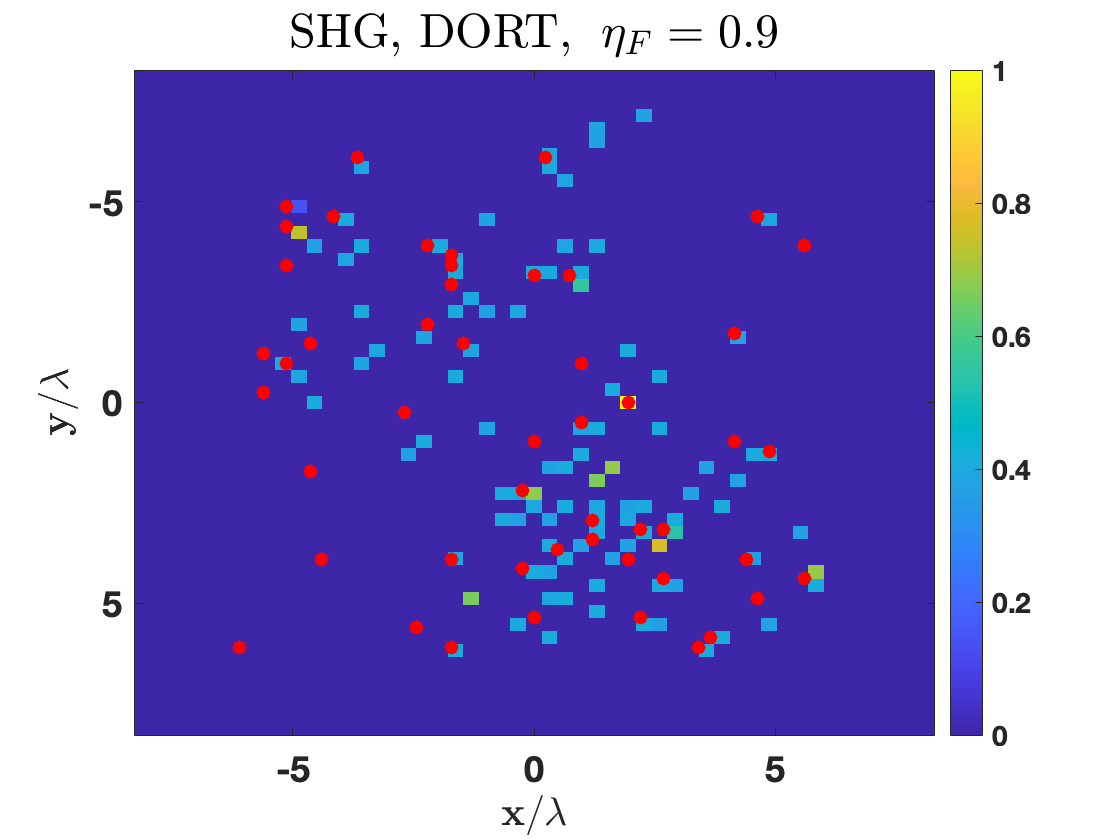} 
\end{center}
\caption{SHG imaging with BSS based on $V$. Upper row: Improved DORT method with $\eta_F=0.5,0.7,0.9$. Lower row: DORT method with $\eta_F=0.5,0.7,0.9$. The red dots represent the exact position of the scatterers.}
\label{fig:SHGV2}

\end{figure}

\begin{figure}[h!]
\begin{center}
  \includegraphics[scale=0.195]{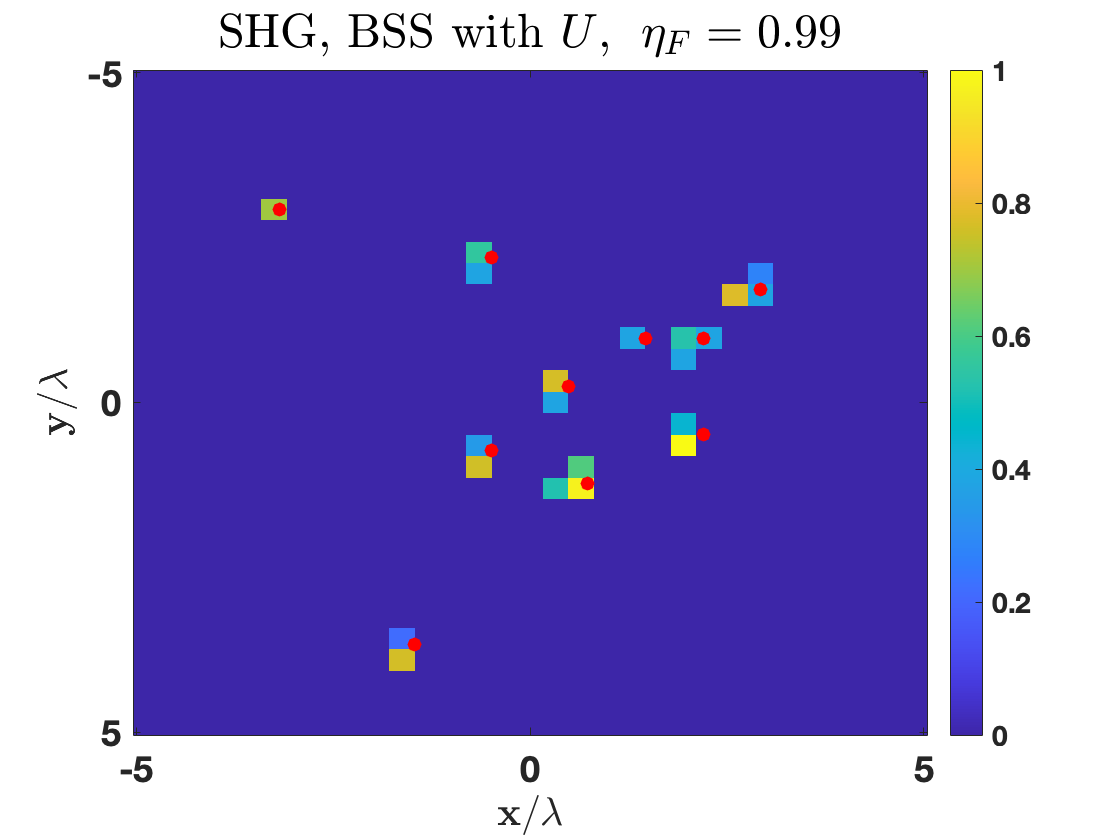}
  \includegraphics[scale=0.195]{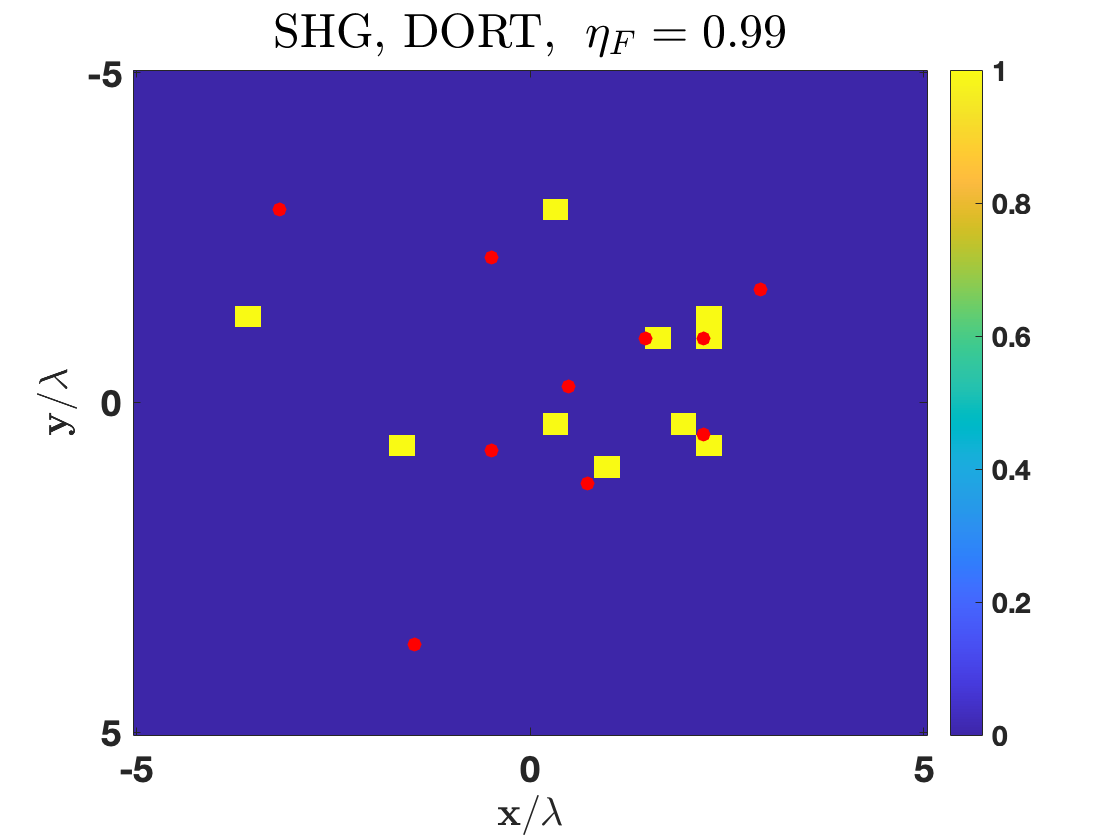} 
\end{center}
\caption{SHG imaging with BSS based on $U$. Left: Improved DORT method. Right: DORT method. The red dots represent the exact position of the scatterers.}
\label{fig:SHGU}
\end{figure}
\subsubsection{Linear scattering imaging} We consider here a classical linear scattering imaging scenario. Recall that it is not possible to use $V$ for the separation since its columns are made of Gaussian random fields and we therefore can only rely on $U$. As in the SHG case, $U$-based separation is more difficult than $V$-based. We choose again a configuration where RobustICA always converges to separated solutions. We set $N_r=500$ and place randomly $N=10$ identical scatterers in a window of size $5 \lambda$. The fluctuations are increased to $k_0 \sigma=3$ in order to lower the size of the isoplanatic patches, which becomes $l_c=3.77\lambda$. The pixel size on the image is now $\lambda/2\rm{NA}$, and to keep the same numerical value as the SHG case we set NA=1.5.


An image of one realization is shown in Figure \ref{fig:linear}. The conclusion is similar to SHG imaging: improved DORT correctly locates the scatterers while DORT does not. As before, we can also consider more complex imaging scenarios that exhibit a (much) greater sensibility to the initial condition of the separation algorithm. 

\begin{figure}[h!]
\begin{center}
  \includegraphics[scale=0.195]{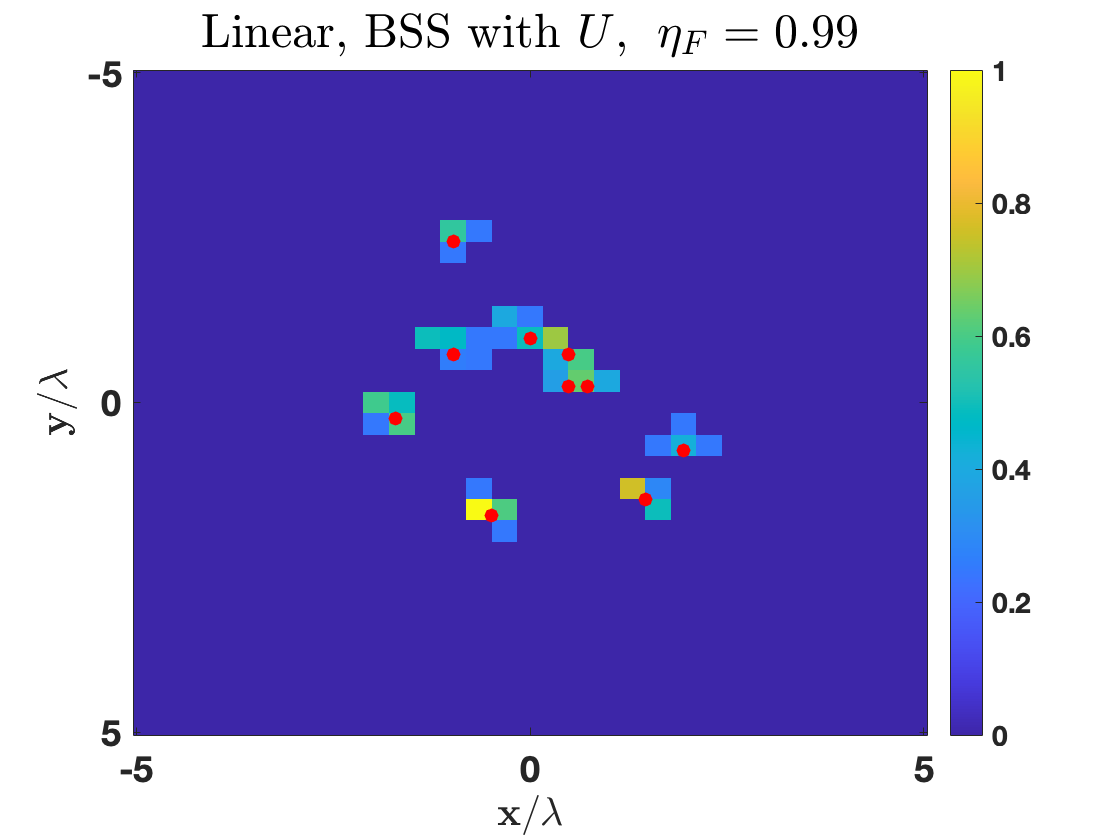}
  \includegraphics[scale=0.195]{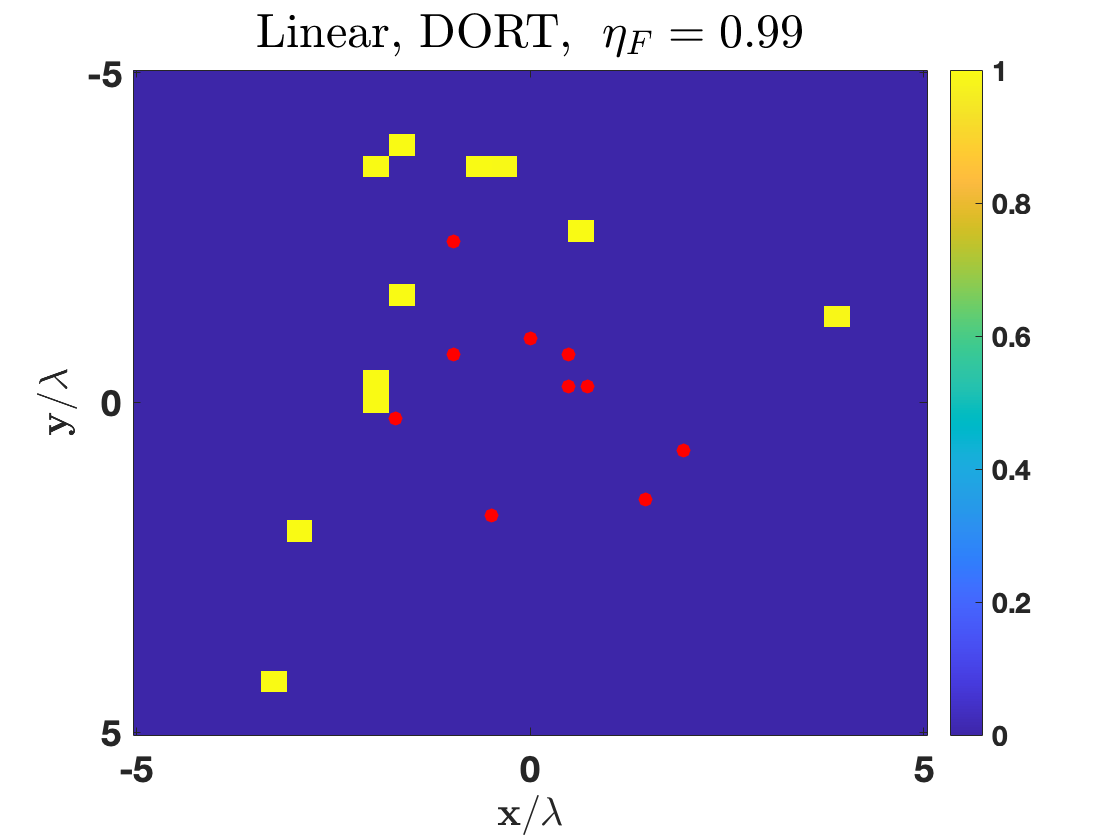} 
\end{center}
\caption{Linear scattering imaging with BSS based on $U$. Left: Improved DORT method. Right: DORT method. The red dots represent the exact position of the scatterers.}
\label{fig:linear}
\end{figure}

\section{Conclusion}
In this work, we have generalized to arbitrary correlated complex-valued sources the standard theoretical separability criteria for blind source separation problems. We have also obtained an error estimate in terms of an explicit quantity involving fourth-order moments of the sources. We verified these separability conditions in the speckle regime and the random geometrical optics regime.  In the speckle case, we exploited the Gaussian character of the random fields, and in the geometrical optics case the fact that fields are perturbations of plane waves. We moreover proposed an imaging technique based on the blind source separation problem that improves upon images obtained with the classical DORT method. The method hinges on the fact that blind source separation provides us with individual Green's functions and not linear combinations.

A few interesting and salient points must be addressed in future works. Because the theoretical result is local and does not ensure the global convergence of numerical algorithms, an important problem is to investigate how the global convergence result of \cite{douglasICA} in the case of real-valued independent sources translates to the case of complex correlated sources. This could in turn help us better understand situations where the algorithms converge to non-separated solutions, as we have observed when using the $U$ matrix for the separation.

From an imaging perspective, the method we propose can operate under stronger scattering conditions than the DORT method, in particular in regimes where the isoplanatic patches are smaller. This will be tested experimentally in the holography setting we described in Section \ref{ssec:mes}. 

\begin{appendix}

  \section{Proof of Theorem \ref{th:err}} \label{app:proof}

  Our goal is to estimate the distance on the complex unit sphere between $w_0$ and $w_e$, which we recall are respectively the normalized first column of $(A^{-1})^*$ and the minimizer of $K$ on the complex unit sphere under the condition $\Re(w,w_0) \geq 0$. 
  We start by establishing some results that will be used throughout the proof, and then move on to first and second order optimality conditions. The error estimate is finally obtained in a conclusion.

  \subsection{Preliminaries} Perturbations of $w_0$ on the feasible set are parametrized by 
$$w(t)=\frac{w_0+t w}{\|w_0+t w\|},$$
where $t \in [-1,1]$, $w$ is an arbitrary vector in $\Cm^N$, and $w_0$ has norm one. Let $\calJ(t):=K(w(t))$. We will choose $w=w_e-w_0$, and expand $\calJ(1)$ as
\be \label{TayJ}
\calJ(1)=\calJ(0)+\calJ'(0)+\frac{1}{2}\calJ''(0)+\frac{1}{3!}\calJ'''(t_*), \qquad t_* \in (0,1).
\ee
We will estimate the r.h.s. from below, and for this need first to compute derivatives of $w(t)$. We denote below $\dot w(t):=dw(t)/dt$, and so on with higher order derivatives. Direct calculations yield
\bea
\dot w(t)&=&\frac{w}{\|w_0+t w\|}-(w_0+ tw) \frac{\Re (w,w_0+tw)}{\|w_0+t w\|^{3}} \label{w1}\\ \nonumber
\ddot w(t)&=&-\frac{2 w \Re (w,w_0+tw)}{\|w_0+t w\|^{3}}-(w_0+tw)\frac{\|w\|^2}{\|w_0+t w\|^{3}}\\
&&+3 (w_0+tw)\frac{(\Re (w,w_0+tw))^2}{\|w_0+t w\|^{5}} \label{w2}\\
\dddot w(t)
&=&-\frac{3 w \|w\|^2}{\|w_0+t w\|^{3}}+\frac{9 w (\Re (w,w_0+tw))^2}{\|w_0+t w\|^{5}}+(w_0+tw)\frac{9\|w\|^2 \Re (w,w_0+tw)}{\|w_0+t w\|^{5}}\nonumber \\
&&-15 (w_0+tw)\frac{(\Re (w,w_0+tw))^3}{\|w_0+t w\|^{7}}. \label{w3}
\eea
The first and second order optimality conditions involve $\dot w(0)$ and $\ddot w(0)$, which, according to \fref{w1}-\fref{w2}, are given by
\bea
\dot w(0)&=&w-w_0\Re (w,w_0). \label{dotzero}\\
\ddot w(0)&=&-2 w \Re (w,w_0)-w_0(w,w)+3w_0(\Re (w,w_0))^2 \label{ddotzero}.
\eea

Given the above expressions, it is natural to single out the contribution of the perturbation vector $w$ along $w_0$, and we write as a consequence $w=u w_0+w_\perp$ for some scalar $u$. Setting $w_\perp$ such that $(w_0,w_\perp)=0$ is a reasonable choice, but we make a different one that seems to overall lead to simpler calculations. Recalling that $C^{(s)}=\E\{s s^*\}$ denotes the correlation matrix of the sources, we choose $w_\perp$ such that $([C^{(s)}]^{-1} w_0,w_\perp)=0$. The motivation is as follows. Since $\E\{x x^*\}=I$, together with the fact that the sources $s$ and the noise $n$ are independent, we notice first that $A C^{(s)} A^*+C^{(n)}=I$, for $C^{(n)}=\E\{n n^*\}$ the correlation matrix of $n$. Some algebra leads to the relation $[C^{(s)}]^{-1} =A^*A+A^*C^{(n)}A[C^{(s)}]^{-1} $, which we rewrite as $[C^{(s)}]^{-1} =A^*A+\calO(C^{(n)})$. Without noise, we hence have simply $[C^{(s)}]^{-1} =A^*A$. Since $C^{(s)}$ is Hermitian, we have equivalently $[C^{(s)}]^{-1} =A A^*+\calO(C^{(n)})$.

Exploiting this last relation then  gives us
\bee
0&=&(w_0, [C^{(s)}]^{-1} w_\perp)=(w_0,A A^* w_\perp)+\calO(C^{(n)})=(A^*w_0,A^* w_\perp)+\calO(C^{(n)})\\
&=&(\gamma e_1,A^* w_\perp)+\calO(C^{(n)}),
\eee
so that
\be \label{errw1A}
(e_1,A^* w_\perp)=\calO(C^{(n)}),
\ee
which will be used extensively in the proof. When there is no noise, this means that $w_\perp$ is orthogonal to $A e_1$.

We will also exploit the relation
\be \label{errC}
[C^{(s)}]^{-1}-(\det C^{(s)})^{-1} I=\calO(C^{(s)}-I),
\ee
which follows from the observation that the off-diagonal terms of $[C^{(s)}]^{-1}$ involve at least one factor $C^{(s)}_{ij}$ for $i \neq j$, so that $[C^{(s)}]^{-1}_{ij}=\calO(\max_{p \neq q}|C^{(s)} _{pq}|)$ when $i \neq j$. Since moreover the diagonal of $C^{(s)}$ is the identity matrix by construction, the diagonal terms of $[C^{(s)}]^{-1}$ are formed by $(\det C^{(s)})^{-1}$ multiplied by sums of terms involving $C^{(s)}_{ii}$ and at least one factor $C^{(s)}_{ij}$ for $i \neq j$. This justifies \fref{errC}. 

We will need to be more precise with the scalar $u$ in the decomposition $w=u w_0+w_\perp$. This is addressed in the next lemma:
\begin{lemma} \label{lemU}
Let $w=w_e-w_0$ and $w=uw_0+w_\perp$. Then, we have $u=u_0+u_1$, where $u_0\geq -1$, $u_0=\calO(C^{(s)}-I)+\calO(\|w_\perp\|)$ and $u_1=\calO(C^{(s)}-I)$.
\end{lemma}
\begin{proof}
Since $w=w_e-w_0$, we have
\bee
u=\frac{([C^{(s)}]^{-1}w,w_0)}{([C^{(s)}]^{-1}w_0,w_0)}&=&-1+\frac{([C^{(s)}]^{-1}w_e,w_0)}{([C^{(s)}]^{-1}w_0,w_0)}\\
&=&-1+\frac{(\det C^{(s)})^{-1} (w_e,w_0)}{([C^{(s)}]^{-1}w_0,w_0)}+\frac{(([C^{(s)}]^{-1}-D_C)w_e,w_0)}{([C^{(s)}]^{-1}w_0,w_0)},
\eee
where $D_C$ is the diagonal matrix $(\det C^{(s)})^{-1} I$. The second term on the r.h.s. above is positive since $(w_e,w_0)=|(\widetilde{w}_e,w_0)|$ by construction and $[C^{(s)}]$ is positive definite. We denote the last term on the r.h.s by $u_1$, and remark that $u_1=\calO(C^{(s)}-I)$ thanks to \fref{errC}. 
Summarizing, we can write $u=u_0+u_1$, where $u_0 \geq -1$ and $u_1=\calO(C^{(s)}-I)$.

We will show in addition that $u_0=\calO(\|w_\perp\|)$, which is obtained as follows:  writing $w_e=(1+u)w_0+w_\perp$, we find,  since both $w_e$ and $w_\perp$ have unit norms, 
$$
1=(1+u_0)^2+2 (1+u_0) \Re(w_0,u_1w_0 +w_\perp)+\|u_1w_0+w_\perp\|^2.
$$
The above quadratic equation on $u_0$ has two roots. Only one verifies the condition $u_0 \geq -1$ for arbitrary $w_\perp$ and $u_2$, and this root admits the expression
\bee
u_0&=&-1-\Re(w_0,u_1 w_0+w_\perp)+\sqrt{1+\Re(w_0,u_1w_0+w_\perp)^2-\|u_1 w_0+w_\perp\|^2},
\eee
where the term under the square root is necessarily positive since $u_0 \geq -1$. Since the square root is of H\"older regularity $1/2$, we find, using that $\sqrt{x+y}\leq \sqrt{x}+\sqrt{y}$ when $x,y \geq 0$,
$$
u_0=\calO(|\Re(w_0,u_1 w_0+w_\perp)|+\|u_1 w_0+w_\perp\|)=\calO(|u_1|+\|w_\perp\|).
$$
This ends the proof since we have seen that  $u_1=\calO(C^{(s)}-I)$.
\end{proof}

\bigskip

We will need in addition to control $\calJ'''(t_*)$, which requires estimates on $w^{(p)}(t_*)$, for $w^{(p)}(t)$ the $p$ derivative of $w(t)$, $p=1,2,3$. Tedious, but not difficult, calculations based on \fref{w1}-\fref{w3} and Lemma \ref{lemU} yield the estimate
\be \label{estimW}
\|w^{(p)}(t) \| = \calO\left(\frac{ \max_{p \neq q}|C^{(s)} _{pq}|^p+\|w_\perp\|^p}{\|w_0+tw\|^{2p+1}} \right).
\ee
We write then $w=w_r+i w_i \in \Cm^N$, and set $\tilde w=(w_r,w_i)^T$ as well as
$$
K_0(\tilde w)=K(w). 
$$
Let
$$
M=\max_{\|\tilde w\|=1} \left(\|\nabla K_0(\tilde w)\|+\|\nabla^2 K_0(\tilde w)\|+\max_{j=1,\cdots,N} \|\partial_{(w_r)_j}\nabla^2 K_0(\tilde w)\|+\max_{j=1,\cdots,N} \|\partial_{(w_i)_j} \nabla^2 K_0(\tilde w)\|\right).
$$
We have $M<\infty$ since $K_0 \in C^{\infty}(\Rm^2)$. Using this, together with \fref{estimW}, more tedious calculations allow us to obtain that
\be \label{estimJ}
\calJ'''(t_*) =\calO\left(\max_{p \neq q}|C^{(s)} _{pq}|^3+\|w_\perp\|^3\right).
\ee

We now move on to the optimality conditions. 

\subsection{First order optimality condition}

We have the following Lemma:

\begin{lemma} \label{opt1} We have:
  \bee
  \left.\frac{d K(w(t))}{dt}\right|_{t=0}&=&\calO\big(\max_{j \neq 1} \big|\E\big\{s_j s_1^* |s_1|^2\big\}-\E\big\{s_j s_1\big\}\E\big\{(s_1^*)^2\big\}\big|\big)\\
  &&+\calO(C^{(s)}-I)+\calO\big(\max_{j}\E\big\{|n_j|^2\big\}\big)+\calO\big(\max_{j}\E\big\{|n_j|^4\big\}\big).
  \eee
\end{lemma}
\begin{proof}
We first introduce the notation $y_0:=w_0^*x=\gamma s_1+w_0^* n$, where $\gamma$ is defined in Section \ref{opti}. Direct calculations starting from \fref{dotzero} yield then
\bee
\left. \frac{1}{4}\frac{d K(w(t))}{dt}\right|_{t=0}&=& \Re\, \E\Big\{(\dot w(0)^* x)  |y_0|^2 y^*_0\Big\}- \Re \,  \E\Big\{(\dot w(0)^* x) y_0\Big\} \E\Big\{(y^*_0)^2\Big\}.
\eee
Injecting the decomposition $w=u w_0+w_\perp$ into $\dot w^*(0)x$, we find the expression
\bea \label{dw0}
\dot w^*(0)x
&=& \eta_1 y_0+\delta w\\ \nonumber
\eta_1&=&(u-\Re u-\Re (w_0,w_\perp))^*\\ \nonumber
\delta w&=&(w_\perp,As)+(w_\perp,n).
\eea
This gives
\bea \label{DK}
\left. \frac{1}{4}\frac{d K(w(t))}{dt}\right|_{t=0}&=& \Re\, \eta_1  \E \big\{|y_0|^4\big\}- \Re\, \eta_1 \E \big\{y_0^2\big\} \E\,\big\{(y^*_0)^2\big\} \nonumber \\
&&+\Re\, \E\big\{ \delta w |y_0|^2 y^*_0\big\}- \Re\, \E\big\{\delta w \, y_0\big\} \E \big\{(y^*_0)^2\big\}. 
\eea
The first two terms in the r.h.s. above are not small, unless $\Re \eta_1$ is. The definition of $\eta_1$ gives $\Re \eta_1=-\Re (w_0,w_\perp)$. 
Since $(w_0,(C^{(s)})^{-1} w_\perp)=0$ by construction, we can write  $w_\perp=C^{(s)} w_0^\perp$, for some $w_0^\perp$ such that $(w_0,w_0^\perp)=0$. Then,  $(w_0,w_\perp)=(w_0,C^{(s)} w^\perp_0)=(w_0,(C^{(s)}-I) w^\perp_0)$. We have hence thus just shown that $\Re \eta_1=\calO(C^{(s)}-I)$, which is small when the sources are weakly correlated. As a consequence,
\be \label{1st0}
\Re\, \eta_1  \E \big\{|y_0|^4\big\}- \Re\, \eta_1 \E \big\{y_0^2\big\} \E\,\big\{(y^*_0)^2\big\}=\calO(C^{(s)}-I).
\ee
We treat now the third and fourth terms in \fref{DK}. We use for this our particular choice for $w_\perp$, that we recall verifies by construction $(A^* w_\perp,e_1)=\calO(C^{(s)}-I)$, see \fref{errw1A}. We have then
\bea \label{relwp}
(w_\perp,As)=(A^* w_\perp,e_1) s_1 +\sum_{j \neq 1}(A^* w_\perp,e_j)s_j=s_1 \calO(C^{(s)}-I)+\sum_{j \neq 1}(A^* w_\perp,e_j) s_j.
\eea 
With the assumption that the noise $n$ and the sources $s$ are mean-zero and independent, we arrive at
\bee
\E\Big\{\delta w |y_0|^2 y^*_0\Big\}&=&\E\Big\{ (w_\perp,As)\Big[|\gamma|^2 \gamma^*|s_1|^2 s_1^*+2 (\Re \gamma^* s_1^* w_0^* n)(w_0^* n )^* +|w_0^* n|^2 \gamma^*s_1^*\Big]\Big\}\\
&&+\E\Big\{ (w_\perp,n) \Big[|\gamma|^2 |s_1|^2 (w_0^*n)^*+2 (\Re \gamma^* s_1^* w_0^*n)s_1^* +|w_0^* n|^2 (w_0^* n)^*\Big]\Big\}.
\eee
Injecting \fref{relwp} in the latter relation, we find after direct algrebra 
\bea \nonumber
\E\big\{ \delta w |y_0|^2 y^*_0\big\}&=&|\gamma|^2 \gamma^* \sum_{j \neq 1}(A^* w_\perp,e_j) \E\Big\{ s_j |s_1|^2 s_1^*\Big\}+\calO(C^{(s)}-I)\\
&&+\calO\big(\max_{j}\E\big\{|n_j|^4\big\}\big)+\calO\big(\max_{j}\E\big\{|n_j|^2\big\}\big). \label{1sto}
\eea
In the same way, we find for the fourth term in \fref{DK},
\bee
\E\Big\{\delta w y_0\Big\}&=&\E\Big\{ (w_\perp,As) \gamma s_1+(w_\perp,n) w_0^*n\Big\}\\
&=&\gamma\sum_{j \neq 1}(A^* w_\perp,e_j) \E\big\{ s_j s_1\big\}+\calO(C^{(s)}-I)+\calO(\max_{j}\E\big\{|n_j|^2\big\}).
\eee
Hence,
\bee
  \E\big\{ \delta w |y_0|^2 y^*_0\big\}- \E\big\{\delta w \, y_0\big\} \E \big\{(y^*_0)^2\big\}
  &=&|\gamma|^2 \gamma^* \sum_{j \neq 1}(A^* w_\perp,e_j) \E\big\{ s_j s_1\big\} \E\big\{(s_1^*)^2\big\}\\
  &&+\calO(C^{(s)}-I)+\calO\big(\max_{j}\E\big\{|n_j|^4\big\}\big)\\
  &&+\calO\big(\max_{j}\E\big\{|n_j|^2\big\}\big),
\eee
which concludes the proof together with \fref{1st0} and \fref{1sto}.
\end{proof}

\subsection{Second order optimality condition}
\begin{lemma} Using notations \fref{not1} and \fref{not2}, we have \label{opt2}
  \bee
  \left. \frac{1}{4}\frac{d^2 K(w(t))}{dt^2}\right|_{t=0}&=& -\|w_\perp\|^2 \left(\E \Big\{(|s_1|^4\Big\}-\E\Big\{s_1^2\Big\} \E\Big\{(s_1^*)^2\Big\}-2\right)\\
  &&+\calO(M^{(s)})+\calO(M^{(n)}).  \eee
\end{lemma}
\begin{proof}
  As in the proof of Lemma \ref{lemU}, we use the notation $y_0:=w_0^*x=\gamma s_1+w_0^* n$, and obtain after direct calculations
  \bea \nonumber
\left.\frac{1}{4}\frac{d^2 K(w(t))}{dt^2}\right|_{t=0}&=&\Re \E \Big\{(\ddot w(0)^* x) y_0^* |y_0|^2\Big\}-\Re \E\Big\{(\ddot w(0)^* x) y_0\Big\} \E\Big\{(y_0^*)^2\Big\}\\\nonumber
&&+2 \E\Big\{|\dot w(0)^* x|^2 |y_0|^2\Big\}-2\E\Big\{(\dot w(0)^* x) y_0\Big\} \E\Big\{(\dot w(0)^T \bar{x}) y_0^*)\Big\}\\ \label{ddw0}
&&+ \Re \E\Big\{(\dot w(0)^* x)^2 (y_0^*)^2\Big\}-\Re \E\Big\{(\dot w(0)^* x)^2\Big\} \E\{(y_0^*)^2\Big\}.
\eea
We need first to explicitate $\ddot w(0)^*x$. Injecting the decomposition $w=uw_0+w_\perp$ in \fref{ddotzero} yields the expression
\begin{align*}
  \ddot w(0)&=\xi^*_1 w_0 - 2 w_\perp \Re(w_0,w_\perp),
\end{align*}
where $\xi_1$ is defined by
              \begin{align*}
                \xi_1=\Big(3 (\Re u)^2+3(\Re (w_0,w_\perp))^2+6 \Re u \Re (w_0,w_\perp) -2 u \Re u - 2 u \Re (w_0,w_\perp)\\
                -|u|^2-\|w_\perp\|^2-2\Re (u w_0,w_\perp)\Big)^*.
              \end{align*}
             Using that $w_0^*x=\gamma s_1+w_0^* n$ and $x=As+n$, we find
              \bee
              \ddot w(0)^*x &=& \gamma \xi_1 s_1 -2 (A^* w_\perp,s)\Re(w_0,w_\perp) +  G_0(n;w_\perp).\\
              G_0(n;w_\perp)&=&\xi_1 (w_0,n) -2 (w_\perp,n) \Re(w_0,w_\perp) .
              \eee
              With \fref{relwp}, we rewrite $\ddot w(0)^*x$ as
\be \label{ddw1}
\ddot w^*(0)x=\gamma \xi_1 s_1+\sum_{j \neq 1}^N \xi_j s_j+G_0(n;w_\perp)+s_1 \calO(C^{(s)}-I),
\ee
where, for $j \neq 1$,
\bee
\xi_j &=&- 2 (A^* w_\perp,e_j) \Re(w_0,w_\perp).
\eee

The next step is to derive a similar expression for $\dot w^*(0)x$. For this, we start from \fref{dw0} and realize that $\eta_1=\calO(C^{(s)}-I)$. We have indeed already seen in the proof of Lemma \ref{opt1} that $\Re \eta_1=\calO(C^{(s)}-I)$, and moreover $\Im \eta_1=\Im u=u_1=\calO(C^{(s)}-I)$ according to Lemma \ref{lemU}. Hence,
\be \label{dw1}
\dot w^*(0)x=\sum_{j\neq 1}^n \eta_j s_j+G_1(n;w_\perp)+s_1 \calO(C^{(s)}-I)
\ee
where for $j \neq 1$,
\bee
\eta_j &=&(A^* w_\perp,e_j)\\
G_1(n;w_\perp)&=&\eta_1 (w_0,n) +(w_\perp,n).
\eee

We are now in position to treat the first two terms in the r.h.s. of \fref{ddw0}. Plugging \fref{ddw1} in these two terms, and using that $(w_0,w_\perp)=\calO(C^{(s)}-I)$ as obtained in the proof of Lemma \ref{opt1}, we find
\begin{align*}
\Re \E \Big\{(\ddot w(0)^* x) y_0^* |y_0|^2\Big\}&-\Re \E\Big\{(\ddot w(0)^* x) y_0\Big\} \E\Big\{(y_0^*)^2\Big\}\\
=&|\gamma|^4 \Re \xi_1 \left(\E \Big\{(|s_1|^4\Big\}-\E\Big\{s_1^2\Big\} \E\Big\{(s_1^*)^2\Big\}\right)+\calO(C^{(s)}-I)\\
  &+\calO\big(\max_{j}\E\big\{|n_j|^2\big\}\big)+\calO\big(\max_{j}\E\big\{|n_j|^4\big\}\big).
\end{align*}
Since $\Re \xi_1=- (\Im u)^2-\|w_\perp\|^2+\calO(C^{(s)}-I)$, and as we have already seen $\Im u=u_1=\calO(C^{(s)}-I)$, we find
\begin{align} \nonumber
\Re \E \Big\{(\ddot w(0)^* x) y_0^* |y_0|^2\Big\}&-\Re \E\Big\{(\ddot w(0)^* x) y_0\Big\} \E\Big\{(y_0^*)^2\Big\}\\\nonumber
=&-|\gamma|^4 \|w_\perp\|^2 \left(\E \Big\{(|s_1|^4\Big\}-\E\Big\{s_1^2\Big\} \E\Big\{(s_1^*)^2\Big\}\right)+\calO(C^{(s)}-I)\\
  &+\calO\big(\max_{j}\E\big\{|n_j|^2\big\}\big)+\calO\big(\max_{j}\E\big\{|n_j|^4\big\}\big). \label{12}
\end{align}

We next move on to the third term in \fref{ddw0}, that verifies, using \fref{dw1},
\bee
\E\Big\{|\dot w(0)^* x|^2 |y_0|^2\Big\} &=&|\gamma|^2\E\Big\{\big|\sum_{j \neq 1} \eta_j s_j\big|^2 |s_1|^2\Big\}+\calO(C^{(s)}-I)\\
&&+\calO\big(\max_{j}\E\big\{|n_j|^2\big\}\big)+\calO\big(\max_{j}\E\big\{|n_j|^4\big\}\big).
\eee
The first term in r.h.s. above is decomposed into
\bee
\E\Big\{\big|\sum_{j \neq 1} \eta_j s_j\big|^2 |s_1|^2\Big\}&=&\E\Big\{\sum_{j \neq 1} |\eta_j s_j|^2 |s_1|^2\Big\}+\E\Big\{\sum_{i \neq 1} \sum_{j \neq i, j \neq 1} \eta_j s_j \eta_i^* s_i^* |s_1|^2\Big\}\\
\E\Big\{\sum_{j \neq 1} |\eta_j s_j|^2 |s_1|^2\Big\}&=&\sum_{j \neq 1 } |\eta_j|^2 \E\{|s_j|^2\}  \E\{|s_1|^2\}\\
&&+\sum_{j \neq 1} |\eta_j|^2 \left(\E\{|s_j|^2 |s_1|^2\}-\E\{|s_j|^2\}  \E\{|s_1|^2\} \right).
\eee
The second term in the r.h.s. above is small when $s_j$ and $s_1$ are weakly correlated. Note that by construction $\E\{|s_j|^2\} = \E\{|s_1|^2\}=1$. For the first term, we need the following relation, obtained using \fref{errw1A},\fref{errC}, and the fact that $\{e_j\}$ forms a basis of $\Rm^N$:
\bee
\sum_{j \neq 1} |\eta_j|^2 &=& \sum_{j \neq 1} |( A^* w_\perp,e_j)|^2\\
&=&(A A^* w_\perp,w_\perp)+\calO(C^{(n)})=((C^{(s)})^{-1}w_\perp,w_\perp)+\calO(C^{(n)})\\
&=&(\det C^{(s)})^{-1} (w_\perp,w_\perp)+\calO(C^{(s)}-I)+\calO(C^{(n)}).
\eee
A short calculation shows that $\det C^{(s)}=1+\calO(C^{(s)}-I)$. Moreover, we deduce from the relation $A C^{(s)} A^*+C^{(n)}=I$ that $C^{(s)}=A^{-1} (A^{-1})^*+\calO(C^{(n)})$. This implies, since $c_1$ is the first column of $(A^{-1})^*$ and $C^{(s)}$ has a unit diagonal, that $\|c_1\|=1+\calO(C^{(n)})$. As a consequence, $|\gamma|=\|c_1\|^{-1}=1+\calO(C^{(n)})$.

Observing finally that
$$
\left|\E\Big\{(\dot w(0)^* x) y_0\Big\}\right|^2=\sum_{i \neq 1}\sum_{j \neq i,  j\neq 1} \eta_j \E\Big\{ s_j s_1 \Big\}\E\Big\{s_i^* s_1^*\Big\}+\calO(C^{(n)}),
$$
we obtain for the third and fourth terms in \fref{ddw0},
\begin{align} \nonumber
  \E\Big\{|\dot w(0)^* x|^2 |y_0|^2\Big\}&-\E\Big\{|\dot w(0)^* x|^2\Big\} \E\Big\{|y_0|^2\Big\}\\ \nonumber
                                         =&-\|w_\perp\|^2+\calO\big(\max_{j}\E\big\{|n_j|^2\big\}\big)+\calO(C^{(s)}-I)+\calO\big(\max_{j}\E\big\{|n_j|^4\big\}\big)
  \\ \nonumber
                                         &+\calO\big( \max_{i \neq j, i\neq 1, j\neq 1}|\E\{s_js_i^* |s_1|^2\}-\E\{s_js_1\} \E\{s_i^* s_1^*\}\big|\big)\\
  &+\calO\big( \max_{j \neq 1} |\E\{|s_j|^2 |s_1|^2\}-\E\{|s_j|^2\} \E\{|s_1|^2\}|\big). \label{34}
\end{align}

We move on now to the fifth and sixth terms in \fref{ddw0} and find
\bee
\E\Big\{(\dot w(0)^* x)^2 (y_0^*)^2\Big\} &=&\gamma^2\E\Big\{\big(\sum_{j \neq 1} \eta_j s_j\big)^2 (s_1^*)^2\Big\}+\calO\big(\max_{j}\E\big\{|n_j|^2\big\}\big)+\calO(C^{(s)}-I)\\
&&+\calO\big(\max_{j}\E\big\{|n_j|^4\big\}\big),
\eee
as well as
\bee
\E\Big\{(\dot w(0)^* x)^2\Big\} \E\Big\{(y_0^*)^2\Big\} &=&\gamma^2\E\Big\{\big(\sum_{j \neq 1} \eta_j s_j\big)^2\Big\} \E\Big\{(s_1^*)^2\Big\}+\calO\big(\max_{j}\E\big\{|n_j|^2\big\}\big)+\calO(C^{(s)}-I)\\
&&+\calO\big(\max_{j}\E\big\{|n_j|^4\big\}\big).
\eee
Writing
\bee
\E\Big\{\big(\sum_{j \neq 1} \eta_j s_j\big)^2 (s^*_1)^2\Big\}&=&\E\Big\{\sum_{j \neq 1} (\eta_j s_j)^2 (s_1^*)^2\Big\}+\E\Big\{\sum_{i \neq 1}\sum_{j \neq i,  j\neq 1}\eta_j s_j \eta_i s_i (s_1^*)^2\Big\}\\
\E\Big\{\big(\sum_{j \neq 1} \eta_j s_j\big)^2\Big\} \E\Big\{(s_1^*)^2\Big\}&=&\E\Big\{\sum_{j \neq 1} (\eta_j s_j)^2\Big\} \E\Big\{(s_1^*)^2\Big\}+\E\Big\{\sum_{i \neq 1}\sum_{j \neq i,  j\neq 1}\eta_j s_j \eta_i s_i \Big\} \E\Big\{(s_1^*)^2\Big\}
\eee
we obtain finally
\begin{align} \nonumber
  \E\Big\{(\dot w(0)^* x)^2 (y_0^*)^2\Big\}&-\E\Big\{(\dot w(0)^* x)^2\Big\} \E\Big\{(y_0^*)^2\Big\}\\ \nonumber
  =&\; \calO\big( \max_{i \neq j, i\neq 1, j\neq 1}|\E\{s_js_i (s_1^*)^2\}-\E\{s_js_i\} \E\{(s_1^*)^2|\}\big)\\ \nonumber
  &+\calO\big( \max_{j \neq 1}|\E\{s_j^2 (s_1^*)^2\}-\E\{s_j^2\} \E\{(s_1^*)^2\}|\big)\\
  &+\calO\big(\max_{j}\E\big\{|n_j|^2\big\}\big)+\calO(C^{(s)}-I)+\calO\big(\max_{j}\E\big\{|n_j|^4\big\}\big)  \label{56}
\end{align}
Gathering \fref{12}-\fref{34}-\fref{56}, together with $\gamma=\|c_1\|^{-1}=1+\calO(C^{(n)})$ then concludes the proof.
\end{proof}

\subsection{Conclusion} We set without lack of generality $\calK(s_1):=\E \{|s_1|^4\}-2-|\E\{s_1^2\}|^2<0$. Combining expansion \fref{TayJ}, Lemmas \ref{opt1} and \ref{opt2}, and estimate \fref{estimJ}, yields
$$
K(w_e)=K(w_0)-2\calK(s_1) \|w_\perp\|^2+\calO(\|w_\perp\|^3)+\calO(M^{(s)})+\calO(M^{(n)}).
$$
Since $w_e$ is a global minimizer of $K$ on the unit sphere with constraint $\Re (w_e,w_0) \geq 0$, we have $K(w_0) \geq K(w_e)$, and therefore
\be \label{ineq}
-2\calK(s_1) \|w_\perp\|^2+\calO(\|w_\perp\|^3)=\calO(M^{(s)})+\calO(M^{(n)}).
\ee
Consider now the polynomial inequality on $x \in \Rm^+$:
$$
x^2(1- ax) \leq b, \qquad a,b \in \Rm^+.
$$
It is not difficult to see that when $b \leq 12/27 a^2$, then $x$ verifies $x \leq 2/3a$, and as a consequence $1-ax \geq 1/3$. Hence, we have the estimate $x \leq \sqrt{3b}$ when  $b \leq 12/27 a^2$.

We apply this to equation \fref{ineq}: there exists then $\eps>0$, such that
$$
M^{(s)}+M^{(n)}\leq \eps
$$
implies that
\be \label{estwp}
w_\perp=\calO\big([M^{(s)}+M^{(n)}]^{1/2}\big).
\ee
We use the latter to control the distance between $w_e$ and $w_0$ on the complex unit sphere that we recall is defined by
$\textrm{dist}(w_e,w_0)=\arccos( \Re (w_e,w_0))$. Simple algebra involving the fact that $w_e-w_0= u w_0+w_\perp$ gives
\bee
\textrm{dist}(w_e,w_0)&=&\arccos( \Re (w_e-w_0,w_0)+1)=\arccos( \Re (u w_0+w_\perp,w_0)+1)\\
&=&\arccos( \Re (w_\perp,w_0)+\Re u+1),
\eee
where by construction $\Re(w_\perp,w_0)+\Re u \in [-2,0]$. The inequality
$$
\arccos (1-x)=\int_0^x \frac{dt}{\sqrt{t(2-t)}} \leq \int_0^x \frac{dt}{\sqrt{t}} = 2 \sqrt{x}
$$
then implies that
$$\textrm{dist}(w_e,w_0) \leq 2 \sqrt{-\Re(w_\perp,w_0)-\Re u}.
$$
We know from Lemma \ref{lemU} that $u=\calO(\|w_\perp\|)+\calO(C^{(s)}-I)$, and from the proof of Lemma \ref{opt2} that $(w_\perp,w_0)=\calO(C^{(s)}-I)$. Together with \fref{estwp}, we finally obtain that
$$
\textrm{dist}(w_e,w_0)=\calO\big([M^{(s)}+M^{(n)}]^{1/4}\big),
$$
which concludes the proof of the theorem.

\end{appendix}


 \bibliographystyle{siam}
  \bibliography{bibliography} 

\begin{thebibliography}{10}

\bibitem{AubryDO}
{\sc A.~Badon, V.~Barolle, K.~Irsch, A.~C. Boccara, M.~Fink, and A.~Aubry},
  {\em Distortion matrix concept for deep optical imaging in scattering media},
  Science Advances, 6 (2020).

\bibitem{BPV25}
{\sc R.~Bartels, O.~Pinaud, and M.~Varughese}, {\em Speckle imaging with blind
  source separation and total variation deconvolution}, Inverse Problems, 41
  (2025), p.~065003.

\bibitem{fastICA}
{\sc E.~Bingham and A.~Hyv\"{a}rinen}, {\em A fast fixed-point algorithm for
  independent component analysis of complex valued signals}, International
  Journal of Neural Systems, 10 (2000), pp.~1--8.

\bibitem{douglasICA}
{\sc S.~Douglas}, {\em On the convergence behavior of the {F}ast{ICA}
  algorithm}, Fourth international conference on Independent Component analysis
  and Blind Signal Separation (ICA2003).

\bibitem{farah2024synthetic}
{\sc Y.~Farah, G.~Murray, J.~Field, M.~Varughese, L.~Wang, O.~Pinaud, and
  R.~Bartels}, {\em Synthetic spatial aperture holographic third harmonic
  generation microscopy}, Optica, 11 (2024), pp.~693--705.

\bibitem{CLASS}
{\sc S.~Kang, P.~Kang, S.~Jeong, Y.~Kwon, T.~Yang, J.~Hong, M.~Kim, K.~Song,
  J.~Park, J.~Lee, M.~Kim, K.~H. Kim, and W.~Choi}, {\em High-resolution
  adaptive optical imaging within thick scattering media using closed-loop
  accumulation of single scattering}, Nature Communications, 8 (2017).

\bibitem{mertz2019introduction}
{\sc J.~Mertz}, {\em Introduction to Optical Microscopy}, Cambridge University
  Press, 2019.

\bibitem{murray2023aberration}
{\sc G.~Murray, J.~Field, M.~Xiu, Y.~Farah, L.~Wang, O.~Pinaud, and
  R.~Bartels}, {\em Aberration free synthetic aperture second harmonic
  generation holography}, Optics Express, 31 (2023), pp.~32434--32457.

\bibitem{GiganTR}
{\sc S.~M. Popoff, A.~Aubry, G.~Lerosey, M.~Fink, A.~C. Boccara, and S.~Gigan},
  {\em Exploiting the time-reversal operator for adaptive optics, selective
  focusing, and scattering pattern analysis}, Phys. Rev. Lett., 107 (2011),
  p.~263901.

\bibitem{DORT2}
{\sc C.~Prada, S.~Manneville, D.~Spoliansky, and M.~Fink}, {\em {Decomposition
  of the time reversal operator: Detection and selective focusing on two
  scatterers}}, The Journal of the Acoustical Society of America, 99 (1996),
  pp.~2067--2076.

\bibitem{zarzoso2009robust}
{\sc V.~Zarzoso and P.~Comon}, {\em Robust independent component analysis by
  iterative maximization of the kurtosis contrast with algebraic optimal step
  size}, IEEE Transactions on neural networks, 21 (2009), pp.~248--261.

\bibitem{GiganNMF}
{\sc L.~Zhu, F.~Soldevila, C.~Moretti, A.~d’Arco, A.~Boniface, X.~Shao,
  H.~Barbosa~de Aguiar, and S.~Gigan}, {\em Large field-of-view non-invasive
  imaging through scattering layers using fluctuating random illumination},
  Nature Communications, 13 (2022).

\end{thebibliography}
 \end{document}